\documentclass[a4paper,11pt]{amsart}    
\usepackage{latexsym, amsmath, amsthm, amssymb, setspace, verbatim}
\usepackage[all]{xy}
\usepackage[utf8]{inputenc} \usepackage[T1]{fontenc} \usepackage{lmodern}
\usepackage{hyperref}
\usepackage{enumitem}

\title{ Strongly self-absorbing \cstar-dynamical systems, II }
\author{ Gábor Szabó }
\address{Westfälische Wilhelms-Universität, Fachbereich Mathematik, \phantom{--------------}\linebreak \text{}\hspace{3.5mm} Einsteinstrasse 62, 48149 Münster, Germany}
\email{ gabor.szabo@uni-muenster.de }
\thanks{\emph{Supported by:} SFB 878 \emph{Groups, Geometry and Actions}. }

\numberwithin{equation}{section}

\begin{document}

\renewcommand\matrix[1]{\left(\begin{array}{*{10}{c}} #1 \end{array}\right)}  
\newcommand\set[1]{\left\{#1\right\}}  

\newcommand{\IC}[0]{\mathbb{C}} 
\newcommand{\IN}[0]{\mathbb{N}}
\newcommand{\IQ}[0]{\mathbb{Q}} 
\newcommand{\IR}[0]{\mathbb{R}}
\newcommand{\IT}[0]{\mathbb{T}}
\newcommand{\IZ}[0]{\mathbb{Z}}

\newcommand{\CC}[0]{\mathcal{C}} 
\newcommand{\CD}[0]{\mathcal{D}}
\newcommand{\CE}[0]{\mathcal{E}} 
\newcommand{\CK}[0]{\mathcal{K}}
\newcommand{\CM}[0]{\mathcal{M}}
\newcommand{\CO}[0]{\mathcal{O}}
\newcommand{\CU}[0]{\mathcal{U}}
\newcommand{\CZ}[0]{\mathcal{Z}}

\renewcommand{\phi}[0]{\varphi}
\newcommand{\eps}[0]{\varepsilon}

\newcommand{\id}[0]{\operatorname{id}}		
\newcommand{\eins}[0]{\mathbf{1}}			
\newcommand{\ad}[0]{\operatorname{Ad}}
\newcommand{\ev}[0]{\operatorname{ev}}
\newcommand{\fin}[0]{{\subset\!\!\!\subset}}
\newcommand{\Hom}[0]{\operatorname{Hom}}
\newcommand{\Aut}[0]{\operatorname{Aut}}
\newcommand{\dimrok}[0]{\dim_{\mathrm{Rok}}}
\newcommand{\dimrokc}[0]{\dim_{\mathrm{Rok}}^{\mathrm{c}}}
\newcommand{\dst}[0]{\displaystyle}
\newcommand{\cstar}[0]{$\mathrm{C}^*$}
\newcommand{\dist}[0]{\operatorname{dist}}
\newcommand{\ann}[0]{\operatorname{Ann}}
\newcommand{\cc}[0]{\simeq_{\mathrm{cc}}}
\newcommand{\scc}[0]{\simeq_{\mathrm{scc}}}

\newcommand{\cf}[2]{cf.\ {\cite[#2]{#1}}}
\renewcommand{\see}[2]{see {\cite[#2]{#1}}}

\newtheorem{satz}{Satz}[section]		
\newtheorem{cor}[satz]{Corollary}
\newtheorem{lemma}[satz]{Lemma}
\newtheorem{prop}[satz]{Proposition}
\newtheorem{theorem}[satz]{Theorem}
\newtheorem*{theoreme}{Theorem}

\theoremstyle{definition}
\newtheorem{defi}[satz]{Definition}
\newtheorem{nota}[satz]{Notation}
\newtheorem{rem}[satz]{Remark}
\newtheorem*{reme}{Remark}
\newtheorem{example}[satz]{Example}
\newtheorem{question}[satz]{Question}


\begin{abstract} 
This is a continuation of the study of strongly self-absorbing actions of locally compact groups on \cstar-algebras. Given a strongly self-absorbing action $\gamma: G\curvearrowright\CD$, we establish permanence properties for the class of separable \cstar-dynamical systems absorbing $\gamma$ tensorially up to cocycle conjugacy. Generalizing results of both Toms-Winter and Dadarlat-Winter, it is proved that the desirable equivariant analogues of the classical permanence properties hold in this context. For the permanence with regard to equivariant extensions, we need to require a mild extra condition on $\gamma$, which replaces $K_1$-injectivity assumptions in the classical theory. This condition turns out to be automatic for equivariantly Jiang-Su absorbing \cstar-dynamical systems, yielding a large class of examples. It is left open whether this condition is redundant for all strongly self-absorbing actions, and we do consider examples that satisfy this condition but are not equivariantly Jiang-Su absorbing.
\end{abstract}

\maketitle

\tableofcontents


\setcounter{section}{-1}

\section{Introduction}

\noindent
This is a continuation of my previous paper \cite{Szabo15ssa}, which introduced (semi-)\-strongly self-absorbing \cstar-dynamical systems and provided an equivariant McDuff-type theorem characterizing equivariant tensorial absorption of such systems, generalizing classical results of Toms-Winter \cite{TomsWinter07} and Kirchberg \cite{Kirchberg04}. The primary motivation comes from the fundamental importance of strongly self-absorbing \cstar-algebras within the Elliott classification program for simple, nuclear \cstar-algebras. Early on, the Cuntz algebras $\CO_2$ and $\CO_\infty$ have arisen as cornerstones of the classification of simple and purely infinite \cstar-algebras, known as the Kirchberg-Phillips classification; see \cite{KirchbergPhillips00, Phillips00, KirchbergC}. The method of localizing the classification of a certain class of \cstar-algebras at a strongly self-absorbing \cstar-algebra, such as the Jiang-Su algebra or a UHF algebra of infinite type, was coined by an influential paper \cite{Winter14Lin} of Winter, and has become a key method in state-of-the-art classification results such as \cite{GongLinNiu15}. It stands to reason that introducing strongly self-absorbing \cstar-dynamical systems could lead to new and significant input into the classification of group actions on \cstar-algebras. This area is of fundamental importance, but surprisingly underdeveloped considering how far the classification of group actions on von Neumann algebras has progressed in comparison. The reader is referred to the introduction of the prequel \cite{Szabo15ssa} and the references therein for a brief historical overview of the significant achievements within these subjects. The reader is in particular recommended to consult Izumi's survey article \cite{Izumi10} in order to get an idea about recent developments in these subjects.

Strongly self-absorbing \cstar-algebras have historically been looked at by example, but for the first time conceptually fleshed out by Toms-Winter in \cite{TomsWinter07}, and later enhanced by Kirchberg's work \cite{Kirchberg04} on central sequence algebras. A significant achievement of these papers was, among other things, to prove highly non-trivial permanence results for $\CD$-stable \cstar-algebras in an abstract and elegant fashion, where $\CD$ is any strongly self-absorbing \cstar-algebra. While the prequel \cite{Szabo15ssa} of this paper is primarily concerned with providing an equivariant McDuff-type theorem and discussing examples, this paper is focused on proving equivariant analogues of the permanence properties that hold in the classical context.
There are some permanence properties that follow, either directly or without too much effort, from the equivariant McDuff-type theorem and general properties of the central sequence algebra. These are treated in the first section and concern the passage to invariant hereditary subsystems, equivariant quotients, equivariantly Morita equivalent systems, or equivariant inductive limits. Some of these were in fact observed already in collaboration with Barlak \cite{BarlakSzabo15} as a consequence of more general observations about sequentially split $*$-homomorphisms between \cstar-dynamical systems. 

As is the case in the classical theory, it poses a far greater challenge to show that the absorption of a (semi-)strongly self-absorbing action is preserved under forming equivariant extensions. This is one of the main objectives of this paper. Since the proof of the classical result is already very technical, it is not surprising that the proof in the equivariant context has to be done with extra care to respect the underlying dynamical structure. Moreover, the proofs in the classical context from \cite{TomsWinter07, Kirchberg04} make use of $K_1$-injectivity in an essential way, or rather that the commutator subgroup of the unitary group of a strongly self-absorbing \cstar-algebra is in the connected component of the unit. This later turned out to be redundant due to Winter's result \cite{Winter11} that every strongly self-absorbing \cstar-algebra is $\CZ$-stable. An equivariant analogue of this unitary commutator condition, which we name unitary regularity, plays a key role from the second section onward. A $G$-action $\alpha$ on a unital \cstar-algebra $A$ is called unitarily regular, if for two unitaries $u,v\in A$, which are approximately fixed by $\alpha$ on a large compact subset of $G$, the unitary commutator $uvu^*v^*\in A$ can be connected to the unit $\eins_A$ by a continuous path that goes entirely through unitaries that are approximately fixed by $\alpha$ on that compact set. As it turns out, a semi-strongly self-absorbing $G$-action $\gamma$ is unitarily regular precisely when the class of all separable, $\gamma$-absorbing $G$-\cstar-dynamical systems is closed under equivariant extensions; see Theorem \ref{ext}. 

Along the way, we will also prove an equivariant analogue of an important result of Dadarlat-Winter \cite{DadarlatWinter09}, which asserts that unital $*$-homomorphisms of the form $\CD\to A\otimes\CD$ are unique up to strong asymptotic unitary equivalence. The equivariant version states that for semi-strongly self-absorbing \cstar-dynamical systems $(\CD,\gamma)$, unital and equivariant $*$-homomorphisms of the form $(\CD,\gamma)\to (A,\alpha)$ are unique up to strong asymptotic $G$-unitary equivalence (cf.\ Definition \ref{sasGue}), if $\alpha$ is strongly cocycle conjugate to $\alpha\otimes\gamma$.
The classical result again makes use of $K_1$-injectivity in an essential way, and we prove the equivariant version under the condition that $\gamma$ is unitarily regular; see Theorem \ref{asue}. In fact, proving this result earlier in the second section comes in handy when applying it to the permanence of $\gamma$-absorption under equivariant extensions, and spares us some of the technical computations that were necessary in \cite{TomsWinter07, Kirchberg04}.

As unitary regularity for semi-strongly self-absorbing actions is essential for proving the more difficult results of this paper, it is natural to ask how abundant this condition is. Unitary regularity turns out to be automatic for equivariantly $\CZ$-stable actions, thus providing a large class of examples. We will also discuss examples that are not equivariantly $\CZ$-stable, and in particular, the equivariant analogue of Winter's result \cite{Winter11} does \emph{not} hold in complete generality. However, all examples of semi-strongly self-absorbing actions that I know of are unitarily regular, and it would not be surprising if it ends up being redundant in this case.

Although the objectives of this paper are of a somewhat technical nature, I expect that further developement of the general structure theory of semi-strongly self-absorbing actions will inevitably bear fruits in applications to classification. For example, it is discovered in joint work with Hirshberg, Winter and Wu in \cite{HirshbergSzaboWinterWu16} that the classical permanence of $\CD$-stability under extensions can be applied in a surprising way to show $\CD$-stability for certain (possibly simple) crossed products by flows. Applications of this spirit in a more dynamical context, using the results presented here, will be pursued in subsequent work.

Part of this research was conducted at the Mathematics Department of the University of Kyoto during a research visit in January 2016, and parts of this paper were completed during a visit to the institute Mittag-Leffler between January and March 2016. I am grateful to both institutes for their kind hospitality and support. I am moreover grateful to Sel{\c c}uk Barlak, Masaki Izumi, Sven Raum, Wilhelm Winter, Jianchao Wu and Joachim Zacharias for some inspiring discussions on the subject of or subjects closely related to this paper.


\section{Preliminaries}

\begin{nota}
Unless specified otherwise, we will stick to the following notational conventions in this paper:
\begin{itemize}
\item $A$ and $B$ denote \cstar-algebras.
\item $\CM(A)$ denotes the multiplier algebra of $A$.
\item $G$ denotes a second-countable, locally compact group.
\item The symbol $\alpha$ is used for a continuous action $\alpha: G\curvearrowright A$. 
\item If $\alpha: G\curvearrowright A$ is an action, then $A^\alpha$ denotes the fixed-point algebra of $A$.
\item If $(X,d)$ is some metric space with elements $a,b\in X$, then we write $a=_\eps b$ as a shortcut for $d(a,b)\leq\eps$.
\item For real numbers $\eps_n\in\IR$, the notation $\eps_n\searrow 0$ shall mean $\eps_n>0$ and that the sequence converges to zero.
\item Given a $\sigma$-compact space $X$ and compact subsets $\varnothing\neq K_n\subset X$, the notation $K_n\nearrow X$ shall mean $K_n\subset K_{n+1}^\circ$ and $X=\bigcup_{n\in\IN} K_n$.
\end{itemize}
\end{nota}

We first recall some needed definitions and notation.

\begin{defi}[cf.~{\cite[3.2]{PackerRaeburn89} and \cite[1.3, 1.6]{Szabo15ssa}}]
\label{def:scc}
Let $\alpha: G\curvearrowright A$ be an action. Consider a strictly continuous map $w: G\to\CU(\CM(A))$.
\begin{enumerate}[label=(\roman*),leftmargin=*] 
\item $w$ is called an $\alpha$-1-cocycle, if one has $w_g\alpha_g(w_h)=w_{gh}$ for all $g,h\in G$.
In this case, the map $\alpha^w: G\to\Aut(A)$ given by $\alpha_g^w=\ad(w_g)\circ\alpha_g$ is again an action, and is called a cocycle perturbation of $\alpha$. Two $G$-actions on $A$ are called exterior equivalent if one of them is a cocycle perturbation of the other.
\item Assume that $w$ is an $\alpha$-1-cocycle. It is called an approximate coboundary, if there exists a sequence of unitaries $x_n\in\CU(\CM(A))$ such that $x_n\alpha_g(x_n^*) \stackrel{n\to\infty}{\longrightarrow} w_g$ in the strict topology for all $g\in G$ and uniformly on compact sets. Two $G$-actions on $A$ are called strongly exterior equivalent, if one of them is a cocycle perturbation of the other via an approximate coboundary.
\item Let $\beta: G\curvearrowright B$ be another action. The actions $\alpha$ and $\beta$ are called cocycle conjugate, written $\alpha\cc\beta$, if there exists an isomorphism $\psi: A\to B$ such that $\psi^{-1}\circ\beta\circ\psi$ and $\alpha$ are exterior equivalent. If $\psi$ can be chosen such that $\psi^{-1}\circ\beta\circ\psi$ and $\alpha$ are strongly exterior equivalent, then $\alpha$ and $\beta$ are called strongly cocycle conjugate, written $\alpha\scc\beta$.
\end{enumerate}
\end{defi}

\begin{rem}
In this paper, it is necessary to make use of (central) sequence algebras and their continuous versions with respect to a given action on the underlying \cstar-algebra, which we shall recall below. We note that we will only need (central) sequence algebras arising from the Fr{\'e}chet filter of cofinite sets, even though one could, in principle, use other filters as well. This is merely a matter of taste; in fact, it should be emphasized that all results in this paper can also be proved by working with ultrapowers, with slightly changed definitions. All the reindexation tricks carried out here will work with an elementary hands-on approach, which is one of the reasons why it is presented using ordinary sequence algebras. In other applications, it might be useful to work with ultrapowers, but then the hands-on reindexation tricks have to be replaced with applications of the so-called $\eps$-test of Kirchberg, see \cite[A.1]{Kirchberg04} or \cite[3.1]{KirchbergRordam12}, which uses the axiom of choice.
\end{rem}

\begin{defi}[cf.~{\cite[1.1]{Kirchberg04} and \cite[1.7, 1.9, 1.10]{Szabo15ssa}}] 
Let $A$ be a \cstar-algebra and $\alpha: G\curvearrowright A$ an action of a locally compact group $G$. 
\begin{enumerate}[label={\textup{(\roman*)}},leftmargin=*]
\item The sequence algebra of $A$ is given as 
\[
A_\infty = \ell^\infty(\IN,A)/\set{ (x_n)_n \mid \lim_{n\to\infty}\| x_n\|=0}.
\]
There is a standard embedding of $A$ into $A_\infty$ by sending an element to its constant sequence. We shall always identify $A\subset A_\infty$ this way, unless specified otherwise.
\item Pointwise application of $\alpha$ on representing sequences defines a (not necessarily continuous) $G$-action $\alpha_\infty$ on $A_\infty$. Let
\[
A_{\infty,\alpha} = \set{ x\in A_\infty \mid [g\mapsto\alpha_{\infty,g}(x)]~\text{is continuous} }
\]
be the continuous part of $A_\infty$ with respect to $\alpha$.
\item For some \cstar-subalgebra $B\subset A_\infty$, the (corrected) relative central sequence algebra is defined as
\[
F(B,A_\infty) = (A_\infty\cap B')/\ann(B,A_\infty).
\]
\item If $B\subset A_\infty$ is $\alpha_\infty$-invariant, then the $G$-action $\alpha_\infty$ on $A_\infty$ induces a (not necessarily continuous) $G$-action $\tilde{\alpha}_\infty$ on $F(B,A_\omega)$. Let
\[
F_\alpha(B,A_\infty) = \set{ y\in F(B,A_\infty) \mid [g\mapsto\tilde{\alpha}_{\infty,g}(y)]~\text{is continuous} }
\]
be the continuous part of $F(B,A_\infty)$ with respect to $\alpha$.
\item In case $B=A$, we write $F(A,A_\infty)=F_\infty(A)$ and $F_\alpha(A,A_\infty)=F_{\infty,\alpha}(A)$.
\end{enumerate}
\end{defi}

\begin{defi}[\cf{Szabo15ssa}{3.1, 4.1}]
Let $\CD$ be a separable, unital \cstar-algebra and $G$ a second-countable, locally compact group. Let $\gamma: G\curvearrowright\CD$ be an action. We say that
\begin{enumerate}[label={(\roman*)},leftmargin=*]
\item $\gamma$ is a strongly self-absorbing action, if the equivariant first-factor embedding
\[
\id_\CD\otimes\eins_\CD: (\CD,\gamma)\to (\CD\otimes\CD,\gamma\otimes\gamma)
\]
is approximately $G$-unitarily equivalent to an isomorphism.
\item $\gamma$ is semi-strongly self-absorbing, if it is strongly cocycle conjugate to a strongly self-absorbing action.
\end{enumerate}
\end{defi}

Let us recall the main result from \cite{Szabo15ssa}, which we will use throughout. Note that we will only treat genuine actions instead of cocycle actions in this paper, so we only cite a special case.

\begin{theorem}[cf.~{\cite[3.7, 4.7]{Szabo15ssa} and \cite[Subsection 4.5]{BarlakSzabo15}}]
\label{equi-McDuff}
Let $G$ be a second-countable, locally compact group. Let $A$ be a separable \cstar-algebra and $\alpha: G\curvearrowright A$ an action. Let $\CD$ be a separable, unital \cstar-algebra and $\gamma: G\curvearrowright\CD$ a semi-strongly self-absorbing action. The following are equivalent:
\begin{enumerate}[label=\textup{(\roman*)},leftmargin=*] 
\item $\alpha$ and $\alpha\otimes\gamma$ are strongly cocycle conjugate. \label{equi-McDuff1}
\item $\alpha$ and $\alpha\otimes\gamma$ are cocycle conjugate.  \label{equi-McDuff2}
\item There exists a unital, equivariant $*$-homomorphism from $(\CD,\gamma)$ to $\big( F_{\infty,\alpha}(A), \tilde{\alpha}_\infty \big)$. \label{equi-McDuff3}
\item The equivariant first-factor embedding $\id_A\otimes\eins: (A,\alpha)\to (A\otimes\CD,\alpha\otimes\gamma)$ is equivariantly sequentially split. \label{equi-McDuff4}
\end{enumerate}
\end{theorem}

\begin{reme}
For the rest of this paper, an action $\alpha$ satisfying condition \ref{equi-McDuff1} from above is called $\gamma$-absorbing or $\gamma$-stable. In the particular case that $\gamma$ is the trivial $G$-action on a strongly self-absorbing \cstar-algebra $\CD$, we will say that $\alpha$ is equivariantly $\CD$-stable.
\end{reme}

Let us first prove or recall some of the permanence properties of $\gamma$-absorbing \cstar-dynamical systems that follow from the above theorem either rather directly or without too much effort.

\begin{rem}
Let $G$ be a second-countable, locally compact group. Let $A$ be a separable \cstar-algebra and $\alpha: G\curvearrowright A$ an action. Let $J\subset A$ be an $\alpha$-invariant ideal and $\pi: A\to A/J$ the quotient map. We have a naturally induced action $\alpha^{\hspace{-2mm}\mod J}: G\curvearrowright A/J$ via $\alpha^{\hspace{-2mm}\mod J}_g(a+J)=\alpha_g(a)+J$ for all $a\in A$, which is well-defined because $J$ is $\alpha$-invariant. Consider the quotient map $\pi_\infty: A_\infty\to (A/J)_\infty$ that is induced on sequence algebra level via componentwise application of $\pi$. Then $\pi_\infty$ is equivariant with respect to $\alpha_\infty$ and $\alpha^{\hspace{-2mm}\mod J}_\infty$. Moreover, it is clear that $\pi_\infty(A_\infty\cap A')\subset (A/J)_\infty\cap (A/J)'$ and $\pi_\infty\big( \ann(A,A_\infty) \big)\subset\ann(A/J, (A/J)_\infty)$, and thus we obtain a well-defined, unital and equivariant $*$-homomorphism $\tilde{\pi}_\infty: F_\infty(A)\to F_\infty(A/J)$. Combining this general observation with \ref{equi-McDuff}, we obtain:
\end{rem}

\begin{cor}
Let $G$ be a second-countable, locally compact group and let $\gamma: G\curvearrowright\CD$ be a semi-strongly self-absorbing action on a separable, unital \cstar-algebra. Let $A$ be a separable \cstar-algebra and $\alpha: G\curvearrowright A$ an action. Let $J\subset A$ be an $\alpha$-invariant ideal. If $\alpha\cc\alpha\otimes\gamma$, then $\alpha^{\hspace{-2mm}\mod J}\cc\alpha^{\hspace{-2mm}\mod J}\otimes\gamma$.
\end{cor}

The following was proved in \cite{BarlakSzabo15} via methods related to equivariantly sequentially split $*$-homomorphisms:

\begin{theorem}[\see{BarlakSzabo15}{Subsection 4.5}]
Let $G$ be a second-countable, locally compact group. Let $A$ be a separable \cstar-algebra and $\alpha: G\curvearrowright A$ an action. Let $\CD$ be a separable, unital \cstar-algebra and $\gamma: G\curvearrowright\CD$ a semi-strongly self-absorbing action. Assume $\alpha\cc\alpha\otimes\gamma$.
\begin{enumerate}[label=\textup{(\roman*)},leftmargin=*]
\item If $E\subset A$ is a hereditary and $\alpha$-invariant \cstar-subalgebra, then $\alpha|_E\cc\alpha|_E\otimes\gamma$.
\item If $\beta: G\curvearrowright B$ is another action on a separable \cstar-algebra such that $(A,\alpha)$ and $(B,\beta)$ are equivariantly Morita equivalent, then $\beta\cc\beta\otimes\gamma$.
\end{enumerate}
\end{theorem}

As it turns out, $\gamma$-absorbing \cstar-dynamical systems are also closed under equivariant inductive limits:

\begin{theorem}
Let $G$ be a second-countable, locally compact group. Let $\CD$ be a separable, unital \cstar-algebra and $\gamma: G\curvearrowright\CD$ a semi-strongly self-absorbing action. Then the class of all separable, $\gamma$-absorbing $G$-\cstar-dynamical systems is closed under equivariant inductive limits.
\end{theorem}
\begin{proof}
Let $(A^{(n)},\alpha^{(n)})$ be a sequence of $\gamma$-absorbing $G$-\cstar-dynamical systems with connecting maps $\phi_n: (A^{(n)},\alpha^{(n)})\to (A^{(n+1)},\alpha^{(n+1)})$. Let $(A,\alpha)$ be the equivariant inductive limit of this system. Then clearly 
\[
(A\otimes\CD,\alpha\otimes\gamma)\cong \lim_{\longrightarrow} \set{ (A^{(n)}\otimes\CD,\alpha^{(n)}\otimes\gamma), \phi_n\otimes\id }.
\] 
Then we have a commutative diagram of the form
\[
\xymatrix@C+3mm{
(A^{(n)},\alpha^{(n)}) \ar[r]^{\phi_n} \ar[d]_{\id\otimes\eins} & (A^{(n+1)},\alpha^{(n+1)}) \ar[d]_{\id\otimes\eins} \ar[r]^{\phi_{n+1,\infty}} & (A,\alpha) \ar[d]_{\id\otimes\eins} \\
(A^{(n)}\otimes\CD,\alpha^{(n)}\otimes\gamma) \ar[r]^{\phi_n\otimes\id} & (A^{(n+1)}\otimes\CD,\alpha^{(n+1)}\otimes\gamma) \ar[r]^{\phi_{n+1,\infty}\otimes\id} & (A\otimes\CD,\alpha\otimes\gamma)
}
\]
for every $n\in\IN$. As the actions $\alpha^{(n)}$ are $\gamma$-absorbing,  the first two downward maps are equivariantly sequentially split by \ref{equi-McDuff}\ref{equi-McDuff4}. It then follows from \cite[3.16]{BarlakSzabo15} that the last downward map is also equivariantly sequentially split. The claim follows from \ref{equi-McDuff}\ref{equi-McDuff2}.
\end{proof}

A substantial part of the rest of this paper is dedicated to proving that under a mild (and possibly redundant) extra condition on $\gamma$, the class of all separable, $\gamma$-absorbing \cstar-dynamical systems is closed under extensions. In the classical setting, this is a highly non-trivial result of Toms-Winter from \cite{TomsWinter07}. Along the way, we will also prove a strengthened uniqueness result for unital and equivariant $*$-homomorphisms of the form $(\CD,\gamma)\to (A,\alpha)$, where $\alpha\scc\alpha\otimes\gamma$, following the example of Dadarlat-Winter \cite{DadarlatWinter09} in the classical theory. We shall use the rest of this preliminary section to gather some technical tools that we need for this purpose.


\begin{lemma}[cf.~{\cite[3.4]{BarlakSzabo15}}]
\label{sequence-lift}
Let $G$ be a second-countable, locally compact group. Let $A$ be a \cstar-algebra and $\alpha: G\curvearrowright A$ an action.
Let $x\in A_{\infty,\alpha}$ and $(x_n)_n\in\ell^\infty(\IN,A)$ a bounded sequence representing $x$. Then for every $g_0\in G$ and $\delta>0$, there exists an open neighbourhood $U$ of $g_0$ such that
\[
\sup_{k\in\IN}~ \sup_{g\in U}~ \|\alpha_g(x_k)-\alpha_{g_0}(x_k) \| \leq \delta.
\]
In particular, the sequence $(x_n)_n$ is a continuous element with respect to the componentwise action of $\alpha$ on $\ell^\infty(\IN,A)$.
\end{lemma}

\begin{lemma} \label{relative-commutant}
Let $G$ be a second-countable, locally compact group. Let $A$ be a (not necessarily separable) \cstar-algebra and $\alpha: G\curvearrowright A$ an action. Let $\CD$ be a separable, unital \cstar-algebra and $\gamma: G\curvearrowright\CD$ a semi-strongly self-absorbing action. Assume $(A,\alpha)\cc (A\otimes\CD, \alpha\otimes\gamma)$. Then for every separable, $\alpha_\infty$-invariant \cstar-subalgebra $B\subset A_{\infty,\alpha}$, there exists a unital and equivariant $*$-homomorphism from $(\CD,\gamma)$ to $F_\alpha(B,A_\infty)$.
\end{lemma}
\begin{proof}
As $\gamma$ is semi-strongly self-absorbing, we know (\cf{Szabo15ssa}{4.6}) that there is a unital and equivariant $*$-homomorphism $\psi: (\CD,\gamma)\to(\CD_{\infty,\gamma}\cap\CD', \gamma_\infty)$. The statement is then a simple reindexation trick, whose details we omit.
\end{proof}

The following is also proved via a simple reindexation trick:

\begin{lemma} \label{reindex1}
Let $G$ be a second-countable, locally compact group. Let $A$ be a $\sigma$-unital \cstar-algebra and $\alpha: G\curvearrowright A$ an action. Let $D$ be a separable, unital \cstar-algebra with an action $\gamma: G\curvearrowright D$. Suppose that there exists a unital and equivariant $*$-homomorphism from $(D,\gamma)$ to $\big( F_{\infty,\alpha}(A), \tilde{\alpha}_\infty \big)$. Then for every separable, $\alpha_\infty$-invariant \cstar-subalgebra $C\subset A_{\infty,\alpha}$, there exists a unital and equivariant $*$-homomorphism from $(D,\gamma)$ to $\big( F_\alpha(C,A_\infty), \tilde{\alpha}_\infty \big)$.
\end{lemma}


In order to prove the main results of this paper, we will need to carefully keep track of certain kinds of unitary homotopies in unital \cstar-algebras.
In particular, given a unital \cstar-dynamical system $(A,\alpha)$ and a unitary $u\in \CU_0(A)$ that is approximately fixed by $\alpha$, we need to consider when it is possible to find a unitary path connecting $u$ with the unit $\eins_A$, which goes entirely through unitaries that are still approximately fixed by $\alpha$. This requires some extra care and gives rise to some technical observations that we shall treat for the rest of this section.

\begin{nota} \label{1-unitaries}
Let $G$ be a second-countable, locally compact group. Let $A$ be a \cstar-algebra and $\alpha: G\curvearrowright A$ an action. For $\eps>0$ and a compact set $K\subset G$, define the closed set
\[
A^\alpha_{\eps,K}=\set{a\in A \mid \|\alpha_g(a)-a\|\leq\eps~\text{for all}~g\in K} \subset A.
\]
If $A$ is unital, then also consider
\[
\CU(A^\alpha_{\eps,K}) = \CU(A)\cap A^\alpha_{\eps,K}
\]
and
\[
\CU_0(A^\alpha_{\eps,K}) = \set{ u(1)\in\CU(A^\alpha_{\eps,K}) \mid u: [0,1]\to\CU(A^\alpha_{\eps,K})~\text{continuous,}~u(0)=\eins}.
\]
\end{nota}

\begin{rem} \label{rem:1-unitaries}
Concerning the above notation for unital $A$, the following observations are trivial, but useful:
\begin{enumerate}[label={(\roman*)},leftmargin=*]
\item One has the product rules
\[
\CU(A^\alpha_{\eps_1,K})\cdot \CU(A^\alpha_{\eps_2,K}) \subset \CU(A^\alpha_{\eps_1+\eps_2,K})
\]
and
\[
\CU_0(A^\alpha_{\eps_1,K})\cdot \CU_0(A^\alpha_{\eps_2,K}) \subset \CU_0(A^\alpha_{\eps_1+\eps_2,K})
\]
for all $\eps_1,\eps_2>0$. \label{1-unitaries:1}
\item It follows from an easy functional calculus argument that for every unitary $u\in\CU(A)$ and $0<\eps<2$, the condition $\|\eins-u\|\leq\eps$ implies $u\in\CU_0(A^\alpha_{2\eps,K})$ for every compact set $K\subset G$. \label{1-unitaries:2}
\item If $\delta>0$, $0<\eps<2$, $K\subset G$ is a compact set and $u\in\CU_0(A^\alpha_{\delta,K})$, then every unitary $v\in\CU(A)$ with $\|v-u\|\leq\eps$ is in $\CU_0(A^\alpha_{\delta+2\eps,K})$. \label{1-unitaries:3}
\end{enumerate}
\end{rem}

\begin{lemma} \label{lifts}
Let $G$ be a second-countable, locally compact group. Let $A$ be a unital \cstar-algebra and $\alpha: G\curvearrowright A$ an action. Let $\eps>0$ and $K\subset G$ a compact set. Then for every unitary $u\in\CU_0( \big(A_{\infty,\alpha}\big)^{\alpha_\infty}_{\eps,K} )\subset A_{\infty,\alpha}$, any representing sequence of unitaries $u_n\in A$ satisfies $u_n\in\CU_0( A^\alpha_{\eps+\delta_n,K})$ for some sequence $\delta_n\searrow 0$.
\end{lemma}
\begin{proof}
Let $\eps>0$ and $K\subset G$ be given, and let $u\in\CU_0( \big(A_{\infty,\alpha}\big)^{\alpha_\infty}_{\eps,K} )$. By \ref{rem:1-unitaries}\ref{1-unitaries:3}, it suffices to show that $u$ admits some representing sequence with this property. Let $\Pi: \ell^\infty(\IN, A)\to A_\infty$ denote the quotient map.
 Let $u:[0,1]\to \CU( \big(A_{\infty,\alpha}\big)^{\alpha_\infty}_{\eps,K} )$ be a unitary path with $u(0)=\eins$ and $u(1)=u$. Applying the unitary lifting theorem (\cf{Blackadar15}{5.1}), we can find a unitary path
\[
w=(w_n)_n: [0,1]\to\CU\big( \Pi^{-1}(A_{\infty,\alpha}) \big) \subset\ell^\infty(\IN, A)
\]
with $w(0)=\eins$ and $\Pi(w(t))=u(t)$ for all $0\leq t\leq 1$.

We have for every $g\in K$ and $0\leq t\leq 1$ that
\begin{equation}  \label{lifts:e1}
\limsup_{n\to\infty} \|\alpha_g(w_n(t))-w_n(t)\|=\|\alpha_{\infty,g}(u(t))-u(t)\| \leq \eps.
\end{equation}
By \ref{sequence-lift}, we know that the elements in $\Pi^{-1}(A_{\infty,\alpha})$ are continuous with respect to the induced action of $\alpha$ on $\ell^\infty(\IN, A)$. In particular, this implies that the map
\[
[0,1]\times G\to\ell^\infty(\IN, A),\quad (t,g)\mapsto \big( \alpha_g(w_n(t)) \big)_n
\]
is uniformly continuous on compact subsets. It thus follows from \eqref{lifts:e1} that we have
\[
\limsup_{n\to\infty}~ \max_{0\leq t\leq 1}~ \max_{g\in K}~\|\alpha_g(w_n(t))-w_n(t)\| \leq \eps.
\]
In particular, we can indeed find some sequence $\delta_n\searrow 0$ such that
\[
\max_{0\leq t\leq 1}~ \max_{g\in K}~\|\alpha_g(w_n(t))-w_n(t)\| \leq \eps+\delta_n\quad\text{for all}~n\in\IN.
\]
The sequence of unitaries $u_n = w_n(1)$ then has the desired properties.
\end{proof}

The following acts as an equivariant analogue of what it means for a unital \cstar-algebra $A$ that the commutator subgroup of $\CU(A)$ is contained in $\CU_0(A)$. It will serve as a replacement for the $K_1$-injectivity assumptions in Toms-Winter's and Dadarlat-Winter's work \cite{TomsWinter07, DadarlatWinter09} on strongly self-absorbing \cstar-algebras when we generalize some of their results to the context of semi-strongly self-absorbing actions:

\begin{defi} \label{unireg}
Let $G$ be a second-countable, locally compact group. Let $A$ be a unital \cstar-algebra and $\alpha: G\curvearrowright A$ an action. We say that $\alpha$ is unitarily regular, if for every compact set $K\subset G$ and $\eps>0$, there exists $\delta>0$ such that $uvu^*v^*\in\CU_0(A^\alpha_{\eps,K})$ for every $u,v\in\CU(A^\alpha_{\delta,K})$.
\end{defi}

We shall observe that all equivariantly $\CZ$-stable actions fall into this class, thus providing a rich class of examples.

\begin{lemma}[cf.\ {\cite[2.6]{jiang}}]
\label{jiang}
Let $B$ be a unital \cstar-algebra and $u\in\CU(B)$ a unitary with trivial $K_1$-class. Then $u\otimes\eins_\CZ\in B\otimes\CZ$ is homotopic to $\eins_{B\otimes\CZ}$.
\end{lemma}

\begin{prop} \label{jiang-su-k1}
Let $G$ be a second-countable, locally compact group. Let $A$ be a unital \cstar-algebra and $\alpha: G\curvearrowright A$ an action. Assume $\alpha\cc\alpha\otimes\id_\CZ$. Then:
\begin{enumerate}[label={\textup{(\roman*)}},leftmargin=*]
\item for every separable, $\alpha_\infty$-invariant \cstar-subalgebra $B\subset A_{\infty,\alpha}$, the fixed-point algebra of the relative commutant $(A_{\infty,\alpha}\cap B')^{\alpha_\infty}$ is $K_1$-injective;\label{jiang-su-k1:1s}
\item the fixed-point algebra $(A_{\infty,\alpha})^{\alpha_\infty}$ is $K_1$-injective; \label{jiang-su-k1:1}
\item the commutator subgroup of $\CU\big( (A_{\infty,\alpha})^{\alpha_\infty} \big)$ is inside $\CU_0\big( (A_{\infty,\alpha})^{\alpha_\infty} \big)$. \label{jiang-su-k1:2}
\item $\alpha$ is unitarily regular. \label{jiang-su-k1:3}
\end{enumerate}
Moreover, the implications \ref{jiang-su-k1:1s}$\implies$\ref{jiang-su-k1:1}$\implies$\ref{jiang-su-k1:2}$\implies$\ref{jiang-su-k1:3} always hold. (without assuming $\alpha\cc\alpha\otimes\id_\CZ$.)
\end{prop}
\begin{proof}
\ref{jiang-su-k1:1s}: Let $u\in (A_{\infty,\alpha}\cap B')^{\alpha_\infty}$ be a unitary with trivial $K_1$-class. Let $C$ be a separable \cstar-subalgebra of $(A_{\infty,\alpha}\cap B')^{\alpha_\infty}$ such that $u\in C$ and $u$ has trivial class in $K_1(C)$. Since we assumed $\alpha\cc\alpha\otimes\id_\CZ$, there exists a unital $*$-homomorphism from $\CZ$ to 
\[
\big( A_{\infty,\alpha}\cap (B\cup C)' \big)^{\alpha_\infty}= \big( A_{\infty,\alpha}\cap B' \big)^{\alpha_\infty}\cap C'
\]
by \ref{relative-commutant}. This gives rise to a commutative diagram of $*$-homomorphisms of the form
\[
\xymatrix{
C \ar[rd]_{\id_{C}\otimes\eins_\CZ} \ar[rr] && (A_{\infty,\alpha}\cap B')^{\alpha_\infty} \\
& C\otimes\CZ \ar[ru] &
}
\]
By \ref{jiang}, it follows that $u\otimes\eins$ is homotopic to $\eins$ in $C\otimes\CZ$. By the above diagram, it thus follows that $u$ is homotopic to $\eins$ in $(A_{\infty,\alpha}\cap B')^{\alpha_\infty}$. 

\ref{jiang-su-k1:1s}$\implies$\ref{jiang-su-k1:1}$\implies$\ref{jiang-su-k1:2}: This is trivial.

\ref{jiang-su-k1:2}$\implies$\ref{jiang-su-k1:3}: Assume that \ref{jiang-su-k1:2} holds and suppose that $\alpha$ is not unitarily regular. We are going to lead this to a contradiction. There exists a compact set $K\subset G$, a number $\eps>0$, a sequence $\delta_n\searrow 0$ and sequences $u_n,v_n\in\CU(A^\alpha_{\delta_n,K})$ such that $u_nv_nu_n^*v_n^*\notin\CU_0(A^\alpha_{\eps,K})$ for all $n$. Since $\delta_n\to 0$, these two sequences of unitaries represent elements $u,v\in\CU\big( (A_{\infty,\alpha})^{\alpha_\infty} \big)$. Their unitary commutator $uvu^*v^*$ is then homotopic to $\eins$ by assumption. In particular, $uvu^*v^*\in\CU_0\big( (A_{\infty,\alpha})^{\alpha_\infty}_{\eps/3,K} \big)$. By \ref{lifts}, the unitary $uvu^*v^*$ can thus be represented by a sequence in $\CU_0(A^{\alpha}_{\eps/2,K})$. So for large enough $n$, the unitaries $u_nv_nu_n^*v_n^*$ are $\eps/4$-close to elements in $\CU_0(A^\alpha_{\eps/2,K})$, and so by \ref{rem:1-unitaries}\ref{1-unitaries:3}, it follows that in fact $u_nv_nu_n^*v_n^*\in\CU_0(A^\alpha_{\eps,K})$. This is a contradiction.
\end{proof}


\section{Strong uniqueness for equivariant $*$-homomorphisms}


\begin{defi} \label{sasGue}
Let $G$ be a second-countable, locally compact group, $A$ and $B$ two unital \cstar-algebras and $\alpha: G\curvearrowright A$ and $\beta: G\curvearrowright B$ two continuous actions. Assume that $A$ is separable. Let $\phi_1, \phi_2: (A,\alpha)\to (B,\beta)$ be two unital and equivariant $*$-homomorphisms. We say that $\phi_1$ and $\phi_2$ are strongly asymptotically $G$-unitarily equivalent, if the following holds:

For every $\eps_0>0$ and compact set $K_0\subset G$, there is a continuous path of unitaries $w: [1,\infty)\to\CU(B)$ satisfying
\begin{equation} \label{sasGue:e1}
w(1)=\eins_B;
\end{equation} 
\begin{equation} \label{sasGue:e2}
\phi_2(x)=\lim_{t\to\infty} w(t)\phi_1(x)w(t)^*\quad\text{for all}~x\in A;
\end{equation}
\begin{equation} \label{sasGue:e3}
\lim_{t\to\infty}~\max_{g\in K}~\|\beta_g(w(t))-w(t)\|=0 \quad\text{for every compact set}~K\subset G;
\end{equation}
\begin{equation} \label{sasGue:e4}
\sup_{t\geq 1}~\max_{g\in K_0}~\|\beta_g(w(t))-w(t)\|\leq\eps_0.
\end{equation}
\end{defi}

\begin{reme}
We note that in the case of compact acting groups, a standard averaging argument allows one deduce that in fact, the unitary path $w: [1,\infty)\to\CU(B)$ from the above definition may be assumed to take values in the fixed-point algebra $B^\beta$.
\end{reme}

\begin{rem}
Let $\gamma: G\curvearrowright\CD$ be a semi-strongly self-absorbing action and $\alpha: G\curvearrowright A$ a $\gamma$-absorbing action on a unital \cstar-algebra. The aim of this section is to prove, under the mild extra condition of unitary regularity on $\gamma$, that unital and equivariant $*$-homomorphisms from $(\CD,\gamma)$ to  $(A,\alpha)$ are unique up to strong asymptotic $G$-unitary equivalence. The basic idea is the same as in \cite{DadarlatWinter09}, but the proof has to be carried out in a way that respects the underlying dynamics. We will begin by first proving weaker uniqueness theorems of this kind. 
\end{rem}

\begin{prop}[see~{\cite[3.5]{Szabo15ssa}}]
\label{aGi-hf}
Let $\gamma: G\curvearrowright\CD$ be a strongly self-absorbing action on a separable, unital \cstar-algebra. Then there exist sequences of unitaries $u_n,v_n\in\CU(\CD\otimes\CD)$ satisfying
\[
\max_{g\in K} \Big( \|u_n-(\gamma\otimes\gamma)_g(u_n)\|+\|v_n-(\gamma\otimes\gamma)_g(v_n)\| \Big) \stackrel{n\to\infty}{\longrightarrow} 0
\]
for every compact set $K\subset G$ and
\[
\ad(u_nv_nu_n^*v_n^*)(d\otimes\eins_\CD) \stackrel{n\to\infty}{\longrightarrow} \eins_\CD\otimes d
\]
for all $d\in\CD$.
\end{prop}

In exactly the same way as in the proof of \cite[4.3]{Szabo15ssa}, this fact carries over to all semi-strongly self-absorbing actions:

\begin{cor}
\label{commutator-unitaries}
Let $(\CD,\gamma)$ be a semi-strongly self-absorbing \cstar-dynamical system. Then there exist sequences of unitaries $u_n,v_n\in\CU(\CD\otimes\CD)$ satisfying
\[
\max_{g\in K} \Big( \|u_n-(\gamma\otimes\gamma)_g(u_n)\|+\|v_n-(\gamma\otimes\gamma)_g(v_n)\| \Big) \stackrel{n\to\infty}{\longrightarrow} 0
\]
for every compact set $K\subset G$ and
\[
\ad(u_nv_nu_n^*v_n^*)(d\otimes\eins_\CD) \stackrel{n\to\infty}{\longrightarrow} \eins_\CD\otimes d
\]
for all $d\in\CD$.
\end{cor}

\begin{nota}
Let $\alpha: G\curvearrowright A$ and $\beta: G\curvearrowright B$ be two actions on unital \cstar-algebras. Let $\eps>0$ and $K\subset G$ a compact set. Let us say that a unital $*$-homomorphism $\phi: A\to B$ is $(\eps,K)$-approximately equivariant, if $\|\phi\circ\alpha_g-\beta_g\circ\phi\|\leq\eps$ for all $g\in K$.
\end{nota}

\begin{prop}
\label{commutator-ue}
Let $G$ be a second-countable, locally compact group. Let $A$ be a unital \cstar-algebra and $\alpha: G\curvearrowright A$ an action. Let $\CD$ be a separable, unital \cstar-algebra and $\gamma: G\curvearrowright\CD$ a semi-strongly self-absorbing action. Assume $\alpha\cc\alpha\otimes\gamma$. 
Let $\phi_1,\phi_2: (\CD,\gamma)\to (A,\alpha)$
be two unital and equivariant $*$-homomorphisms. Then there exist sequences of unitaries $u_n,v_n\in\CU(A)$ satisfying
\[
\max_{g\in K} \Big( \|u_n-\alpha_g(u_n)\|+\|v_n-\alpha_g(v_n)\| \Big) \stackrel{n\to\infty}{\longrightarrow} 0
\]
for every compact set $K\subset G$ and
\[
\ad(u_nv_nu_n^*v_n^*)\circ\phi_1 \stackrel{n\to\infty}{\longrightarrow} \phi_2
\]
in point-norm.
\end{prop}
\begin{proof}
This is completely analogous to the proof of \cite[3.4(iii)]{Szabo15ssa}.
\end{proof}

Let us also record the following approximate uniqueness theorem, which will have its uses in the later sections:

\begin{lemma}
\label{approx-uniqueness}
Let $G$ be a second-countable, locally compact group. Let $A$ be a unital \cstar-algebra and $\alpha: G\curvearrowright A$ an action. Let $\CD$ be a separable, unital \cstar-algebra and $\gamma: G\curvearrowright\CD$ a semi-strongly self-absorbing action. Assume $\alpha\cc\alpha\otimes\gamma$. Then:
\begin{enumerate}[label=\textup{(\roman*)},leftmargin=*] 
\item any two unital and equivariant $*$-homomorphisms $\phi^{(1)},\phi^{(2)}: (\CD,\gamma)\to (A_{\infty,\alpha},\alpha_\infty)$ are $G$-unitarily equivalent; \label{approx-uniqueness:1}
\item for every $\eps>0$, finite set $F\fin\CD$ and compact set $K_0\subset G$, there exists $\delta>0$ and a compact set $K_1\subset G$ satisfying:

Whenever $\phi^{(1)}, \phi^{(2)}: \CD\to A$ are unital, $(\delta,K_1)$-approximately equivariant $*$-homomorphisms, then there exists $u\in\CU(A^\alpha_{\eps,K_0})$ with $\phi^{(1)}(x)=_\eps u\phi^{(2)}(x)u^*$ for all $x\in F$. \label{approx-uniqueness:2}
\end{enumerate}
\end{lemma}
\begin{proof}
\ref{approx-uniqueness:1}: Let $C$ be a separable, $\alpha_{\infty}$-invariant \cstar-subalgebra of $A_{\infty,\alpha}$ containing the images of these $*$-homomorphisms. By \ref{relative-commutant}, we can find a unital and equivariant $*$-homomorphism from $(\CD,\gamma)$ to the relative commutant $(A_{\infty,\alpha}\cap C', \alpha_\infty)$. But then, we have an induced commutative diagram of $*$-homomorphisms of the form
\[
\xymatrix{
(\CD,\gamma) \ar[rd]_{\phi^{(i)}\otimes\eins} \ar[rr]^{\phi^{(i)}} & & (A_{\infty,\alpha},\alpha_\infty) \\
& (C\otimes\CD, (\alpha_\infty|_C)\otimes\gamma) \ar[ru] &
}
\]
for $i=1,2$, where the map in the upwards direction is given by multiplication. By \ref{commutator-ue}, it follows that $\phi^{(1)}$ and $\phi^{(2)}$ are approximately $G$-unitarily equivalent. Since the target of these $*$-homomorphisms is a (\cstar-dynamical system on a) sequence algebra, a simple reindexation argument shows that they are $G$-unitarily equivalent.

\ref{approx-uniqueness:2}: Let us assume that the statement does not hold. We are going to lead this to a contradiction. There exist certain $\eps>0$, $F\fin\CD$ and $K_0\subset G$ and sequences $\delta_n\searrow 0$ and $K_n\nearrow G$, and moreover sequences $\phi^{(1)}_n, \phi^{(2)}_n: \CD\to A$ of unital, $(\delta_n,K_n)$-approximately equivariant $*$-homomorphisms such that the above conclusion does not hold. As $\delta_n\to 0$ and $G=\bigcup_{n\in\IN} K_n$, we have two unital and equivariant $*$-homomorphisms
\[
\phi^{(1)}=(\phi^{(1)}_n)_n,\ \phi^{(2)}=(\phi^{(2)}_n)_n: (\CD,\gamma)\to (A_{\infty,\alpha},\alpha_\infty).
\]
By \ref{approx-uniqueness:1}, we can find a unitary $u\in\CU\big( (A_{\infty,\alpha})^{\alpha_\infty} \big)$ with $\phi^{(2)}(x) = u\phi^{(1)}(x)u^*$ for all $x\in\CD$. Let $u_n\in A$ be a sequence of unitaries representing $u$. By \cite[2.4]{Szabo15ssa} (or essentially the same argument as in the proof of \ref{lifts}), one has $u_n\in\CU(A^\alpha_{\eps,K_0})$ for all sufficiently large $n$. Moreover, it follows directly from the definition of the $\phi^{(i)}$ that we also have
$\psi^{(2)}_n(x)=_\eps u_n\psi^{(1)}_n(x)u_n^*$ for all $x\in F$ and for sufficiently large $n$. But this gives a contradiction.
\end{proof}

\begin{prop}
\label{approx-inverse}
Let $G$ be a second-countable, locally compact group. Let $\CD$ be a separable, unital \cstar-algebra and $\gamma: G\curvearrowright\CD$ a semi-strongly self-absorbing action. Let $\alpha: G\curvearrowright A$ be an action on a unital \cstar-algebra with $\alpha\scc\alpha\otimes\gamma$. For every $\eps>0$, $F\fin A$ and compact set $K\subset G$, there exists a unital $*$-homomorphism $\psi: A\otimes\CD\to A$ that is $(\eps,K)$-approximately equivariant and satisfies $\psi(a\otimes\eins)=_\eps a$ for all $a\in F$.
\end{prop}
\begin{proof}
First consider the case $(A,\alpha)=(\CD,\gamma)$ and that $\gamma$ is strongly self-absorbing. There exists an isomorphism $\phi: (\CD\otimes\CD,\gamma\otimes\gamma) \to (\CD,\gamma)$. By \ref{commutator-ue}, the equivariant $*$-homomorphism $\phi\circ(\id_\CD\otimes\eins): (\CD,\gamma)\to(\CD,\gamma)$ is approximately $G$-inner, say via a sequence $u_n\in\CU(A^\alpha_{\eps_n,K_n})$ with $\eps_n\searrow 0$ and $K_n\nearrow G$. Then $\psi_n=\ad(u_n)\circ\phi$ satisfies the desired property for sufficiently large $n$.

Now consider the case where $\gamma$ is semi-strongly self-absorbing. Since $\gamma$ is strongly cocycle conjugate to a strongly self-absorbing action, the assertion follows by applying \cite[4.2]{Szabo15ssa} and a standard $\eps/3$-argument.

For general $(A,\alpha)$, use the previous case and choose a sequence of unital $*$-homomorphism $\psi_n: \CD\otimes\CD\to\CD$ that are $(\eps_n,K_n)$-approximately equivariant for some $\eps_n\searrow 0$ and $K_n\nearrow G$ and satisfy $\psi_n(x\otimes\eins)\to x$ for all $x\in\CD$. Then the sequence of $*$-homomorphisms $\id_A\otimes\psi_n: A\otimes\CD\otimes\CD\to A\otimes\CD$ are also $(\eps_n,K_n)$-approximately equivariant and satisfy $\psi_n(a\otimes x\otimes\eins)\to a\otimes x$ for all $a\in A$ and $x\in\CD$. Since $\alpha\scc\alpha\otimes\gamma$, the claim now follows again by applying \cite[4.2]{Szabo15ssa} and a standard $\eps/3$-argument.
\end{proof}

\begin{cor}
\label{lift-from-D}
Let $G$ be a second-countable, locally compact group. Let $A$ be a unital \cstar-algebra and $\alpha: G\curvearrowright A$ an action. Let $\CD$ be a separable, unital \cstar-algebra and $\gamma: G\curvearrowright\CD$ a semi-strongly self-absorbing action. Assume $\alpha\scc\alpha\otimes\gamma$. Let $\phi: (\CD,\gamma)\to (A_{\infty,\alpha}, \alpha_\infty)$ be a unital and equivariant $*$-homomorphism. Then there exists a sequences $\eps_n\searrow 0$ and $K_n\nearrow G$ and a sequence of $*$-homomorphisms $\phi_n: \CD\to A$ that lifts $\phi$ and such that $\phi_n$ is $(\eps_n, K_n)$-approximately equivariant for all $n$.
\end{cor}
\begin{proof}
By \ref{approx-inverse}, we can in particular construct some sequence of $*$-homomor\-phisms $\psi_n: \CD\to A$ such that each $\psi_n$ is $(\delta_n,K_n^{(1)})$-approximately equivariant with respect to some sequences $\delta_n\searrow 0$ and $K_n^{(1)}\nearrow G$. Then this defines a unital and equivariant $*$-homomorphism $\psi=(\psi_n)_n: \CD\to (A_{\infty,\alpha},\alpha_\infty)$. By \ref{approx-uniqueness}\ref{approx-uniqueness:1}, there is a unitary $u\in\CU\big( (A_{\infty,\alpha})^{\alpha_\infty} \big)$ with $\phi=\ad(u)\circ\psi$. Let $u_n\in\CU(A)$ be unitaries representing $u$. By \ref{lifts}, we have $u_n\in\CU(A^\alpha_{\eta_n, K_n^{(2)}})$ for some sequences $\eta_n\searrow 0$ and $K_n^{(2)}\nearrow G$. Then $\phi_n=\ad(u_n)\circ\psi_n: \CD\to A$ is a sequence of $*$-homomorphisms lifting $\phi$ and each $\phi_n$ is $(\delta_n+2\eta_n, K_n^{(1)}\cap K_n^{(2)})$-approximately equivariant for every $n$. As we have $\delta_n+2\eta_n\searrow 0$ and $K_n^{(1)}\cap K_n^{(2)}\nearrow G$, this finishes the proof. 
\end{proof}

\begin{lemma}
\label{almost-asue1}
Let $G$ be a second-countable, locally compact group. Let $A$ be a unital \cstar-algebra and $\alpha: G\curvearrowright A$ an action that is unitarily regular. Let $\CD$ be a separable, unital \cstar-algebra and $\gamma: G\curvearrowright\CD$ a semi-strongly self-absorbing action. Assume $\alpha\cc\alpha\otimes\gamma$. Let $\phi_1,\phi_2: (\CD,\gamma)\to (A,\alpha)$
be two unital and equivariant $*$-homomorphisms. 

Then for every $\eps>0$, $F\fin\CD$ and compact set $K\subset G$, there exists a unitary $w\in\CU_0(A^\alpha_{\eps,K})$ with $w\phi_1(x)w^*=_\eps \phi_2(x)$ for all $x\in F$.
\end{lemma}
\begin{proof}
Let $\eps>0$, $F\fin\CD$ and compact set $K\subset G$ be given. Choose a $\delta>0$ satisfying the property in \ref{unireg} with regard to $(\eps,K)$.
By \ref{commutator-ue}, there exist unitaries $u,v\in \CU(A^\alpha_{\delta,K})$ such that $\phi_2(x)=_\eps\ad(uvu^*v^*)\circ\phi_1(x)$ for all $x\in F$. Then $w=uvu^*v^*\in\CU_0(A^\alpha_{\eps,K})$ does the trick.
\end{proof}

\begin{lemma}
\label{almost-asue2}
Let $G$ be a second-countable, locally compact group. Let $A$ be a unital \cstar-algebra and $\alpha: G\curvearrowright A$ an action. Let $\CD$ be a separable, unital \cstar-algebra and $\gamma: G\curvearrowright\CD$ a semi-strongly self-absorbing action that is unitarily regular. Assume $\alpha\cc\alpha\otimes\gamma$. Let $\phi_1,\phi_2: (\CD,\gamma)\to (A,\alpha)$
be two unital and equivariant $*$-homomorphisms. 

Then for every $\eps>0$, $F\fin\CD$ and compact set $K\subset G$, there exists a unitary $w\in\CU_0(A^\alpha_{\eps,K})$ with $w\phi_1(x)w^*=_\eps \phi_2(x)$ for all $x\in F$.
\end{lemma}
\begin{proof}
This proof will follow the same line of argument as in \cite[3.4(iii)]{Szabo15ssa}.
Let $\eps>0$, $F\fin\CD$ and $K\subset G$ be given. Since $\gamma\otimes\gamma: G\curvearrowright\CD\otimes\CD$ is semi-strongly self-absorbing, \ref{almost-asue1} applies to the flip automorphism on $\CD\otimes\CD$ and the identity map. This implies that the flip automorphism can be approximately implemented by unitaries in $\CU_0\big( (\CD\otimes\CD)^{\gamma\otimes\gamma}_{\eps, K} \big)$. Consider a unitary $v\in \CU_0\big( (\CD\otimes\CD)^{\gamma\otimes\gamma}_{\eps, K} \big)$ with
\[
v(x\otimes\eins)v^*=_\eps\eins\otimes x~\text{for all}~x\in F.
\]
Since we assumed $\alpha\cc\alpha\otimes\gamma$, we can find a unital and equivariant $*$-homomorphism $\psi: (\CD,\gamma)\to (A_{\infty,\alpha}\cap A', \alpha_\infty)$. This induces two equivariant $*$-homomorphisms
\[
\kappa_1,\kappa_2: (\CD\otimes\CD,\gamma\otimes\gamma)\to (A_{\infty,\alpha},\alpha_\infty)~\text{via}~\kappa_j(a\otimes b)=\phi_j(a)\cdot \psi(b),~j=1,2.
\]
The unitary 
\[
w=\kappa_2(v^*)\cdot\kappa_1(v)\in \CU_0( \big(A_{\infty,\alpha}\big)^{\alpha_\infty}_{\eps,K} )\cdot\CU_0( \big(A_{\infty,\alpha}\big)^{\alpha_\infty}_{\eps,K} ) \subset \CU_0( \big(A_{\infty,\alpha}\big)^{\alpha_\infty}_{2\eps,K} )
\]
then satisfies
\[
\begin{array}{cll}
w\phi_1(x)w^* &=& \kappa_2(v^*)\kappa_1(v)\kappa_1(x\otimes\eins)\kappa_1(v^*)\kappa_2(v) \\
&=& \kappa_2(v^*)\kappa_1\big( v(x\otimes\eins)v^* \big) \kappa_2(v) \\
&=_\eps& \kappa_2(v^*)\kappa_1(\eins\otimes x)\kappa_2(v) \\
&=& \kappa_2(v^*)\kappa_2(\eins\otimes x)\kappa_2(v) \\
&=& \kappa_2\big( v^*(\eins\otimes x)v \big) \\
&=_{\eps}& \kappa_2(x\otimes\eins) ~=~ \phi_2(a)
\end{array}
\]
for all $g\in K$ and $x\in F$. By \ref{lifts}, we may represent $w$ by a sequence $w_n\in\CU_0\big(A^\alpha_{3\eps, K} \big)$. For sufficiently large $n$, we then have $w_n\phi_1(x)w_n^*=_{3\eps}\phi_2(x)$ for all $x\in F$, which shows our claim.
\end{proof}

The following will serve as an equivariant version of a basic homotopy lemma:

\begin{lemma}[\cf{DadarlatWinter09}{2.1}]
\label{dw-lemma}
Let $G$ be a second-countable, locally compact group. Let $\CD$ be a separable, unital \cstar-algebra and $\gamma: G\curvearrowright\CD$ a semi-strongly self-absorbing action. For every $\eps>0$, $F_1\fin\CD$ and compact set $K\subset G$, there exist $\delta>0$ and $F_2\fin\CD$ with the following property:

Let $A$ be a unital \cstar-algebra and $\alpha: G\curvearrowright A$ an action with $\alpha\scc\alpha\otimes\gamma$. Let $\psi: \CD\to A$
be a unital $*$-homomorphism that is $(\delta, K)$-equivariant, and let $u: [0,1]\to\CU(A^\alpha_{\delta,K})$ be a unitary path satisfying
\[
u(0)=\eins \quad\text{and}\quad \|[u(1), \psi(x)]\|\leq\delta\quad\text{for all}~ x\in F_2.
\] 
Then there exists a unitary path $w: [0,1]\to\CU(A^\alpha_{\eps,K})$ satisfying
\[
w(0)=\eins_A,\quad w(1)= u(1),\quad \max_{0\leq t\leq 1}~\|[w(t), \psi(x)]\|\leq\eps
\] 
for all $x\in F_1$. Moreover, we may choose $w$ in such a way that
\[
\|w(t_1)-w(t_2)\| \leq \|u(t_1)-u(t_2)\|\quad\text{for all}\quad 0\leq t_1, t_2\leq 1.
\]
\end{lemma}
\begin{proof}
Let $\eps>0$, $F_1\fin\CD$ and $K\subset G$ be given. Let $\eta=\eps/15$. By \ref{aGi-hf}, the action $\gamma$ has approximately $G$-inner half-flip, so choose a unitary $v\in\CU(\CD\otimes\CD)$ with 
\begin{equation} \label{dw-lemma:e1}
\max_{g\in K}~ \|v-(\gamma\otimes\gamma)_g(v)\|\leq\eta \quad\text{and}\quad \eins_\CD\otimes x =_\eta v(x\otimes\eins_\CD)v^*
\end{equation} 
for all $x\in F_1$. Choose contractions $s_1,\dots,s_k,t_1,\dots,t_k\in\CD$ with 
\begin{equation} \label{dw-lemma:e2}
v=_\eta\sum_{i=1}^k s_i\otimes t_i.
\end{equation} 
Then set
\begin{equation} \label{dw-lemma:e3}
\delta=\eta/k=\eps/15k\quad\text{and}\quad F_2=\set{s_1,\dots,s_k}.
\end{equation}
Now let $(A,\alpha)$, $\psi: \CD\to A$ and $u: [0,1]\to\CU(A^{\alpha}_{\delta,K})$ be as in the assertion.
Apply \ref{approx-inverse} and choose a unital and equivariant $*$-homomorphism $\kappa: A\otimes\CD\to A$ that is $(\eta, K)$-equivariant and satisfies 
\begin{equation} \label{dw-lemma:e4}
\kappa(a\otimes\eins)=_\eta a\quad\text{for all}~a\in \psi(F_1)\cup\set{u(1)}.
\end{equation}  
We define a new unitary path $w: [0,1]\to\CU(A)$ via
\[
w(t) = \kappa\Big( (\psi\otimes\id_\CD)(v)^*\cdot(u(t)\otimes\eins_\CD) \cdot(\psi\otimes\id_\CD)(v)\Big).
\]
We have for all $g\in K$ and $0\leq t\leq 1$ that
\[
\begin{array}{ccl}
\alpha_g(w(t)) &=_\eta& \kappa\circ(\alpha\otimes\gamma)_g\Big( (\psi\otimes\id_\CD)(v)^*\cdot(u(t)\otimes\eins_\CD) \cdot(\psi\otimes\id_\CD)(v) \Big) \\
&\hspace{4mm}\stackrel{\eqref{dw-lemma:e1}}{=}_{\hspace{-2mm}2\eta+2\delta}& \kappa\Big( (\psi\otimes\id_\CD)(v)^*\cdot(\alpha_g(u(t))\otimes\eins_\CD) \cdot(\psi\otimes\id_\CD)(v) \Big) \\
& =_\delta &  \kappa\Big( (\psi\otimes\id_\CD)(v)^*\cdot(u(t)\otimes\eins_\CD) \cdot(\psi\otimes\id_\CD)(v) \Big) \\
& \hspace{-2mm}=& w(t)
\end{array}
\]
Moreover, we have
\[
w(0) = \kappa\Big( (\psi\otimes\id_\CD)(v)^*\cdot(u(0)\otimes\eins_\CD) \cdot(\psi\otimes\id_\CD)(v) \Big) = \eins_A
\]
and
\[
\begin{array}{ccl}
w(1) &=& \kappa\Big( (\psi\otimes\id_\CD)(v)^*\cdot(u(1)\otimes\eins_\CD) \cdot(\psi\otimes\id_\CD)(v) \Big) \\

&\hspace{-1mm}\stackrel{\eqref{dw-lemma:e2}}{=}_{\hspace{-2mm}\eta}& \dst \kappa\Big( (\psi\otimes\id_\CD)(v)^*\cdot(u(1)\otimes\eins_\CD) \cdot\sum_{i=1}^k \psi(s_i)\otimes t_i \Big) \\

&\hspace{1mm}\stackrel{\eqref{dw-lemma:e3}}{=}_{\hspace{-2.5mm}k\cdot\delta}& \dst \kappa\Big( (\psi\otimes\id_\CD)(v)^*\cdot\sum_{i=1}^k \psi(s_i)\otimes t_i\cdot(u(1)\otimes\eins_\CD) \Big) \\

&\stackrel{\eqref{dw-lemma:e2}}{=}_{\hspace{-2mm}\eta}& \kappa(u(1)\otimes\eins_\CD) ~\stackrel{\eqref{dw-lemma:e4}}{=}_{\hspace{-2mm}\eta}~ u(1).
\end{array}
\]
Lastly, we have for all $x\in F_1$ and $0\leq t\leq 1$ that
\[
\begin{array}{ccl}
w(t)\psi(x) &\stackrel{\eqref{dw-lemma:e4}}{=}_{\hspace{-2mm}\eta}& \kappa\Big( (\psi\otimes\id_\CD)(v)^*\cdot(u(t)\otimes\eins_\CD) \cdot(\psi\otimes\id_\CD)(v) \Big)\cdot \kappa\big( \psi(x)\otimes\eins_\CD \big) \\
&=& \kappa\Big( (\psi\otimes\id_\CD)(v)^*\cdot (u(t)\otimes\eins_\CD)\cdot (\psi\otimes\id_\CD)(v(x\otimes\eins_\CD)) \Big) \\
&\stackrel{\eqref{dw-lemma:e1}}{=}_{\hspace{-2mm}\eta}& \kappa\Big( (\psi\otimes\id_\CD)(v)^*\cdot (u(t)\otimes\eins_\CD)\cdot (\psi\otimes\id_\CD)((\eins_\CD\otimes x)v) \Big) \\
&=& \kappa\Big( (\psi\otimes\id_\CD)(v^*(\eins_\CD\otimes x))\cdot (u(t)\otimes\eins_\CD)\cdot (\psi\otimes\id_\CD)(v) \Big) \\
&\stackrel{\eqref{dw-lemma:e1}}{=}_{\hspace{-2mm}\eta}& \kappa\Big( (\psi\otimes\id_\CD)((x\otimes\eins_\CD)v^*)\cdot (u(t)\otimes\eins_\CD)\cdot (\psi\otimes\id_\CD)(v) \Big) \\
&\stackrel{\eqref{dw-lemma:e4}}{=}_{\hspace{-2mm}\eta}& \psi(x)w(t).
\end{array}
\]
Combining these calculations with the definition of $\eta=\eps/15$ and $\delta=\eta/k$, we get that indeed
\[
w(0)=\eins_A,\quad w(1)=_{\eps/3} u(1),\quad \max_{0\leq t\leq 1}~\max_{g\in K}~\|\alpha_g(w(t))-w(t)\|\leq\eps/3,
\]
and $\dst\max_{0\leq t\leq 1}~\|[w(t), \psi(x)]\|\leq\eps/3$ for all $x\in F_1$. Since we may assume $\eps<1$, we have $\|\eins-w(1)^*u(1)\|\leq 1/3$. By functional calculus, we can write $w(1)^*u(1)=\exp(ia)$ for a self-adjoint element $a\in A$. Again by functional calculus, the unitary path $v: [0,1]\to\CU(A), t\mapsto \exp(ita)$ satisfies $v(0)=\eins$, $v(1)=w(1)^*u(1)$ and $\|v(t)-\eins\|\leq \|v(1)-\eins\|\leq\eps/3$ for all $0\leq t\leq 1$. By considering the extended unitary path
\[
[0,2]\ni t\mapsto \begin{cases} w(t) &,\quad 0\leq t\leq 1 \\
w(1)v(t-1) &,\quad 1\leq t\leq 2 \end{cases}
\]
and rescaling, it is thus clear how to get a new unitary path $w: [0,1]\to\CU(A)$ satisfying the desired properties. (Note that by simple application of the triangle inequality, the $\eps/3$-approximations become $\eps$-approximations for the new path.)
\end{proof}

\begin{rem}
The very last part of the statement of \ref{dw-lemma} will be mostly irrelevant for proving the main results of this paper. However, we record a useful technical consequence for future use. The reader is advised to skip the following Lemma if he or she is only interested in the main result of this section.
\end{rem}

\begin{lemma} \label{homotopy commutant}
Let $G$ be a second-countable, locally compact group. Let $\CD$ be a separable, unital \cstar-algebra and $\gamma: G\curvearrowright\CD$ a strongly self-absorbing action. Let $A$ be a unital \cstar-algebra and $\alpha: G\curvearrowright A$ an action with $\alpha\scc\alpha\otimes\gamma$. Let $\psi: (\CD,\gamma)\to (A_{\infty,\alpha}, \alpha_\infty)$ be a unital and equivariant $*$-homomorphism. Then
\[
\CU_0\Big( \big( A_{\infty,\alpha}\cap\psi(\CD)' \big)^{\alpha_\infty} \Big) = \CU_0\big( (A_{\infty,\alpha})^{\alpha_\infty} \big) \cap \psi(\CD)'.
\]
In other words, a unitary in the fixed-point algebra $\big( A_{\infty,\alpha}\cap\psi(\CD)' \big)^{\alpha_\infty}$ is homotopic to $\eins$ precisely when it is homotopic to $\eins$ inside the larger fixed-point algebra $(A_{\infty,\alpha})^{\alpha_\infty}$.
\end{lemma}
\begin{proof}
The inclusion from left to right is trivial. So let 
\[
u\in\CU\Big( \big( A_{\infty,\alpha}\cap\psi(\CD)' \big)^{\alpha_\infty} \Big)
\] 
be a unitary and let $u: [0,1]\to\CU\big( (A_{\infty,\alpha})^{\alpha_\infty} \big)$ be a unitary path with $u(0)=\eins$ and $u(1)=u$. Let $\Pi: \ell^\infty(\IN,A)\to A_\infty$ be the quotient map. Applying the unitary lifting theorem (\cf{Blackadar15}{5.1}), we can find a unitary path
\[
w=(w_n)_n: [0,1]\to\CU\big( \Pi^{-1}(A_{\infty,\alpha}) \big) \subset\ell^\infty(\IN, A)
\]
with $w(0)=\eins$ and $\Pi(w(t))=u(t)$ for all $0\leq t\leq 1$.

We have for every $g\in G$ and $0\leq t\leq 1$ that
\begin{equation}  \label{homot:e1}
\limsup_{n\to\infty} \|\alpha_g(w_n(t))-w_n(t)\|=\|\alpha_{\infty,g}(u(t))-u(t)\| = 0.
\end{equation}
By \ref{sequence-lift}, we know that the elements in $\Pi^{-1}(A_{\infty,\alpha})$ are continuous with respect to the induced action of $\alpha$ on $\ell^\infty(\IN, A)$. In particular, this implies that the map
\[
[0,1]\times G\to\ell^\infty(\IN, A),\quad (t,g)\mapsto \big( \alpha_g(w_n(t)) \big)_n
\]
is uniformly continuous on compact subsets. It thus follows from \eqref{homot:e1} that we have
\begin{equation} \label{homot:e2}
\limsup_{n\to\infty}~ \max_{0\leq t\leq 1}~ \max_{g\in K}~\|\alpha_g(w_n(t))-w_n(t)\| = 0
\end{equation}
for all compact sets $K\subset G$. Also denote the function $f: [0,1]\times [0,1]\to [0,\infty)$ given by $f(t_1,t_2)=\sup_{n\in\IN} \|w_n(t_1)-w_n(t_2)\|$, which is continuous by assumption and vanishes on the diagonal.

Next, it follows from \ref{lift-from-D} that there exist sequences $\eta_n\searrow 0$ and $K_n\nearrow G$ such that $\psi$ lifts to a sequence of $*$-homomorphisms $\psi_n: \CD\to A$ so that $\psi_n$ is $(\eta_n, K_n)$-equivariant. Let $\eps_j\searrow 0$ be a sequence and $F_j\fin\CD$ an increasing sequence with dense union. For every $j\in\IN$, apply \ref{dw-lemma} and choose $\delta_j>0$ and $H_j\fin\CD$ satisfying the given property for the triple $(\eps_j, K_j, F_j)$. For every $j\in\IN$, we can apply \eqref{homot:e2} together with the approximate centrality of the sequence $w_n(1)$ to find a strictly increasing sequence $n_j\in\IN$ so that
\[
\max_{0\leq t\leq 1}~ \max_{g\in K}~\|\alpha_g(w_n(t))-w_n(t)\| \leq \delta_j
\]
and
\[
\max_{x\in H_j} \|[w_n(1),\psi_n(x)]\|\leq\delta_j,\quad \eta_n\leq\delta_j
\]
for all $n\geq n_j$. Then by choice of the involved constants and \ref{dw-lemma}, we can find unitary paths $v_n: [0,1]\to\CU(A)$ satisfying
\begin{equation} \label{homot:e3}
v_n(0)=\eins,\quad v_n(1)=w_n(1)
\end{equation}
\begin{equation} \label{homot:e4}
\|v_n(t_1)-v_n(t_2)\|\leq f(t_1,t_2)\quad\text{for all}~0\leq t_1, t_2\leq 1
\end{equation}
and
\begin{equation} \label{homot:e5}
\max_{0\leq t\leq 1}~ \max_{g\in K_j}~\|\alpha_g(v_n(t))-v_n(t)\| \leq \eps_j
\end{equation}
\begin{equation} \label{homot:e6}
\max_{0\leq t\leq 1}~\max_{x\in F_j}~ \|[v_n(t),\psi_n(x)]\|\leq\eps_j
\end{equation}
for all $n_j\leq n<n_{j+1}$ and $j\in\IN$.

Now it follows from \eqref{homot:e4} and the continuity of $f$ that the map
\[
v: [0,1]\to\CU(A_\infty),\quad t\mapsto \Pi\big( (v_n(t))_n \big)
\]
is a continuous unitary path. It follows from \eqref{homot:e3} that $v(1)=u$. By the choices of $\eps_j$, $F_j\fin\CD$ and $K_j\subset G$, it follows from \eqref{homot:e5} and \eqref{homot:e6} that $v(t)\in \big(A_{\infty,\alpha}\cap\psi(\CD)'\big)^{\alpha_\infty}$ for all $0\leq t\leq 1$. This finishes the proof.
\end{proof}

Now comes the main result of this section, which is a strengthened uniqueness theorem for equivariant $*$-homomorphsms from $(\CD,\gamma)$. For the statement, recall the notion of strong cocycle conjugacy; see \ref{def:scc}.

\begin{theorem}
\label{asue}
Let $G$ be a second-countable, locally compact group. Let $A$ be a unital \cstar-algebra and $\alpha: G\curvearrowright A$ an action. Let $\CD$ be a separable, unital \cstar-algebra and $\gamma: G\curvearrowright\CD$ a semi-strongly self-absorbing action. Assume $\alpha\scc\alpha\otimes\gamma$, and that either $\gamma$ or $\alpha$ is unitarily regular. Then any two unital and equivariant $*$-homomorphisms $\phi_1,\phi_2: (\CD,\gamma)\to (A,\alpha)$ are strongly asymptotically $G$-unitarily equivalent.
\end{theorem}
\begin{proof}
Let $\eps_0>0$ and $K_0\subset G$ a compact set.
Let $K_0\subset K_n \subset K_{n+1}^\circ\subset G$ be an increasing sequence of compact sets with $G=\bigcup_{n\in\IN} K_n$. Choose a decreasing sequence $\eps_0>\eps_1>\eps_2>\dots$ converging to zero. Let $F_n\fin\CD$ be an increasing sequence of finite sets with dense union. For every $n\geq 1$ and every triple $(\eps_n, F_n, K_n)$, choose a pair $(\delta_n, F_n')$ satisfying the assumptions in \ref{dw-lemma}, with $(\eps_n,F_n,K_n)$ in place of $(\eps, F_1, K)$ and $(\delta_n, F_n')$ in place of $(\delta, F_2)$. By inductively making the $\eps_n$ smaller and the $F_n$ bigger, if necessary, we may assume $F_{n+1}=F_n'$ and $\eps_{n+1}\leq\delta_n/2$ for all $n$.

By applying either \ref{almost-asue1} or \ref{almost-asue2}, we can choose a sequence of unitary paths $v_n: [0,1]\to\CU(A)$ satisfying
\begin{equation} \label{asue:e1}
v_n(0)=\eins_A,\quad \max_{0\leq t\leq 1}~\max_{g\in K_n}~\|\alpha_g(v_n(t))-v_n(t)\|\leq\eps_n
\end{equation}
and
\begin{equation} \label{asue:e2}
\quad v_n(1)\phi_1(x)v_n(1)^*=_{\eps_n} \phi_2(x)
\end{equation}
for all $x\in F_n$.

Now observe that for all $n\geq 2$, the map $\psi_n=\ad(v_n(1))\circ\phi_1: \CD\to A$ is an $(2\eps_n, K_n)$-approximately equivariant $*$-homomorphism, and we have a unitary path $u_{n+1}=v_{n+1}v_n^*: [0,1]\to\CU(A)$ satisfying $u_{n+1}(0)=\eins_A$,
\[
\begin{array}{ccl}
\|\alpha_g(u_{n+1}(t))-u_{n+1}(t)\| &\leq& \|\alpha_g(v_n(t))-v_n(t)\|+\|\alpha_g(v_{n+1}(t))-v_{n+1}(t)\| \\
&\stackrel{\eqref{asue:e1}}{\leq}& \eps_n+\eps_{n+1}
\end{array}
\]
for all $g\in K_n$ and $0\leq t\leq 1$, and moreover
\[
\begin{array}{ccl}
\|[u_{n+1}(1),\psi_n(x)]\| &=& \|[v_{n+1}(1)v_n(1)^*,v_n(1)\phi_1(x)v_n(1)^*]\|\\
&=&\|(\ad(v_{n+1}(1))\circ\phi_1)(x)-(\ad(v_{n}(1))\circ\phi_1)(x)\| \\
&\stackrel{\eqref{asue:e2}}{\leq}& \eps_n+\eps_{n+1}
\end{array}
\]
for all $x\in F_n$. Since we have chosen the sequence $\eps_n$ so that $\eps_n+\eps_{n+1}\leq 2\eps_n\leq \delta_{n-1}$, we can apply \ref{dw-lemma} and obtain a unitary path $w_{n+1}: [0,1]\to\CU(A)$ satisfying
\begin{equation} \label{asue:e3}
w_{n+1}(0)=\eins_A,\quad w_{n+1}(1)= u_{n+1}(1)=v_{n+1}(1)v_n(1)^*;
\end{equation}
\begin{equation} \label{asue:e4} 
\max_{0\leq t\leq 1}~\max_{g\in K_n}~\|\alpha_g(w_{n+1}(t))-w_{n+1}(t)\|\leq\eps_{n-1};
\end{equation}
\begin{equation} \label{asue:e5}
\max_{0\leq t\leq 1}~\|[w_{n+1}(t), \psi_n(x)]\|\leq\eps_{n-1}\quad\text{for all}~x\in F_{n-1}.
\end{equation}
Let us now consider the unitary path $w: [1,\infty)\to\CU(A)$ given by
\[
 w(t)=\begin{cases} v_2(t-1) &,\quad\text{if}~ 1\leq t\leq 2 \\
w_{n+1}(t-n)v_n(1) &,\quad\text{if}~n\leq t\leq n+1~\text{for some}~n\geq 2.
\end{cases}
\]
It follows from \eqref{asue:e3} that this map is indeed well-defined and continuous. We are going to show that this path satisfies the desired properties from \ref{sasGue}. As $w(1)=\eins_A$ is clear, \eqref{sasGue:e1} follows.
 Moreover, we have for every $x\in\bigcup_{n\in\IN} F_n$ that
\[
\begin{array}{cl}
\multicolumn{2}{l}{ \dst\limsup_{t\to\infty}~ \|\phi_2(x)-(\ad(w(t))\circ\phi_1)(x)\| } \\
\leq& \dst\limsup_{n\to\infty}~\Big( \|\phi_2(x)-(\ad(v_n(1))\circ\phi_1)(x)\| \\
& \dst +\max_{0\leq t\leq 1}~ \|(\ad(v_n(1))\circ\phi_1)(x)-(\ad(w_{n+1}(t)v_n(1))\circ\phi_1)(x)\| \Big) \\
\stackrel{\eqref{asue:e2}}{\leq} & \dst \limsup_{n\to\infty}~ \eps_n + \max_{0\leq t\leq 1} \|[w_{n+1}(t),\psi_n(x)] \| \\
\stackrel{\eqref{asue:e5}}{\leq} & \dst \limsup_{n\to\infty}~ \eps_n+\eps_{n-1} ~=~ 0.
\end{array}
\]
Since the union $\bigcup_{n\in\IN} F_n$ is a dense subset of $\CD$, it follows that indeed $\phi_2=\lim_{t\to\infty} \ad(w(t))\circ\phi_1$ in point-norm. This verifies \eqref{sasGue:e2}. Furthermore, we observe that every compact subset $K$ is eventually contained in one of the sets $K_n$ for sufficiently large $n$, and thus
\[
\begin{array}{cl}
\multicolumn{2}{l}{ \dst\limsup_{t\to\infty}~ \max_{g\in K}~ \|\alpha_g(w(t))-w(t)\| }\\
\leq& \dst\limsup_{n\to\infty}~\max_{0\leq t\leq 1}~\max_{g\in K}~ \|w_{n+1}(t)v_n(1)-\alpha_g(w_{n+1}(t)v_n(1))\| \\
\stackrel{\eqref{asue:e1}, \eqref{asue:e4}}{\leq} & \dst\limsup_{n\to\infty}~\eps_n+\eps_{n-1} ~=~ 0.
\end{array}
\]
This verfies \eqref{sasGue:e3}. Lastly, keeping in mind $K_0\subset K_n$ for all $n$, we have
\[
\begin{array}{cl}
\multicolumn{2}{l}{ \dst\sup_{t\geq 1}~ \max_{g\in K_0}~ \|\alpha_g(w(t))-w(t)\| }\\
=& \dst \max_{0\leq t\leq 1}~ \max_{g\in K_0}~\max\Big( \|\alpha_g(v_2(t))-v_2(t)\| ~, \\
& \dst\sup_{n\geq 2}~ \|w_{n+1}(t)v_n(1)-\alpha_g(w_{n+1}(t)v_n(1))\| \Big) \\
\stackrel{\eqref{asue:e1}, \eqref{asue:e4}}{\leq} & \dst\sup_{n\geq 2}~\eps_n+\eps_{n+1} ~\leq~ \eps_0.
\end{array}
\]
This verfies \eqref{sasGue:e4}. Since $K_0\subset G$ and $\eps_0>0$ were arbitrary parameters, the claim follows.
\end{proof}

\begin{rem}
Let $G$ be a second-countable, compact group and $\gamma: G\curvearrowright\CD$ a strongly self-absorbing, unitarily regular action on a separable, unital \cstar-algebra. Let $\alpha: G\curvearrowright A$ be an action on a unital \cstar-algebra. Let $X=\Hom_G(\CD,A\otimes\CD)$ be the compact space of unital, $G$-equivariant $*$-homomorphisms from $(\CD,\gamma)$ to $(A\otimes\CD,\alpha\otimes\gamma)$, equipped with the point-norm topology. Consider the unital and equivariant $*$-homomorphism $\phi: (\CD,\gamma)\to\big( \CC(X)\otimes A\otimes\CD), \id\otimes\alpha\otimes\gamma \big)$ given by $\phi(d)(\kappa)=\kappa(d)$. Let $\kappa_0:=\eins_A\otimes\id_\CD\in X$.
Then \ref{asue} implies that $\phi$ is strongly asymptotically $G$-unitarily equivalent to $\eins\otimes\kappa_0$. Since $G$ is compact, the path of unitaries may be chosen to have image in the fixed-point algebra $(\CC(X)\otimes A\otimes\CD)^{\id\otimes\alpha\otimes\gamma}$. In particular, $\phi$ and $\eins\otimes\kappa_0$ are homotopic as $G$-equivariant $*$-homomorphisms, which shows that $X$ contracts to the point $\set{\kappa_0}$. This can be used to make statements about the $KK^G$-theory for strongly self-absorbing \cstar-dynamical systems, similar to \cite[Section 3]{DadarlatWinter09}. 
It is unclear whether one can make such statements about the $KK^G$-theory for non-compact $G$. In the key step above, the application of \ref{asue} would then only give a homotopy that goes through $(\eps,K)$-approximately equivariant $*$-homomorphisms for arbitrarily small $\eps>0$ and large compact sets $K\subset G$. 
\end{rem}


\section{Equivariant $\sigma$-ideals}

\noindent
In this section, we will generalize Kirchberg's notion of a $\sigma$-ideal  \cite[1.5]{Kirchberg04} to the equivariant context. Note that for $\IZ$-actions, this has been already introduced by Liao in \cite[5.5]{Liao15} as a technical tool, and the notion of $G$-$\sigma$-ideals presented in this section coincides with his definition for $G=\IZ$.

\begin{defi}
Let $G$ be a second-countable, locally compact group. Let $A$ be a \cstar-algebra with an action $\alpha: G\curvearrowright A$. An $\alpha$-invariant ideal $J\subset A$ is called a $G$-$\sigma$-ideal, if for every separable, $\alpha$-invariant \cstar-subalgebra $C\subset A$, there exists a positive contraction $e\in (J\cap C')^\alpha$ such that $ec=c$ for all $c\in J\cap C$.
\end{defi}

As is to be expected, the typical examples for $G$-$\sigma$-ideals arise in the context of \cstar-dynamical systems induced on (central) sequence algebras. In this section, we consider some classes of examples in such a context.

\begin{lemma}[see {\cite[1.4]{Kasparov88}}]
\label{fixed-approx-unit}
Let $A$ be a $\sigma$-unital \cstar-algebra, $G$ a $\sigma$-compact, locally compact group and $\alpha: G\curvearrowright$ an action. Let $J\subset A$ be a separable, $\alpha$-invariant ideal. Then there exists an approximate unit $(e_n)_{n \in \IN}$ for $J$ such that $\max_{g\in K}~\|\alpha_g(e_n)-e_n\|\to 0$ for every compact subset $K\subset G$, and such that $(e_n)_n$ is quasicentral relative to $A$.
\end{lemma}

\begin{example} \label{ideal-sequences}
Let $G$ be a second-countable, locally compact group. Let $A$ be a \cstar-algebra with an action $\alpha: G\curvearrowright A$. Let $J\subset A$ be an $\alpha$-invariant ideal. Then $J_{\infty,\alpha}\subset A_{\infty,\alpha}$ is a $G$-$\sigma$-ideal.
\end{example}
\begin{proof}
Let $C\subset A_{\infty,\alpha}$ be a separable, $\alpha_\infty$-invariant \cstar-subalgebra. We choose a dense sequence $c_n\in C$ and a dense sequence in $b_n\in C\cap J_{\infty,\alpha}$. Let each $c_n$ and $b_n$ be represented by a bounded sequence $c^{(n)}_k\in A$ and $b^{(n)}_k\in A$, respectively. Applying \ref{fixed-approx-unit}, we may choose a countable approximate unit $h_n\in C\cap J_{\infty,\alpha}$ with $\max_{g\in K} \|\alpha_{\infty,g}(h_{n})-h_{n}\|\to 0$ for every compact set $K\subset G$, and which is quasicentral relative to $C$. Let each $h_n$ be represented by a bounded sequence $h^{(n)}_k\in J$. Choose a sequence $K_n\nearrow G$. For every $j\in\IN$, there exists $h_{n_j}$ such that
\[
\max_{1\leq n\leq j}~ \Big( \|[c_n,h_{n_j}]\|+\|b_n-b_nh_{n_j}\| \Big) \leq 1/j
\]
and
\[
\max_{g\in K_j}~ \|\alpha_{\infty,g}(h_{n_j})-h_{n_j}\| \leq 1/j.
\]
Using \cite[3.5]{BarlakSzabo15}, we can find a strictly increasing sequence $k_j\in\IN$ such that
\[
\max_{1\leq n\leq j}~\sup_{k\geq k_j} \Big( \|[c^{(n)},h^{(n_j)}_k]\|+\|b^{(n)}_k-b^{(n)}_kh^{(n_j)}_k\| \Big) \leq 2/j 
\]
and
\[
\sup_{k\geq k_j}~\max_{g\in K_j}~ \|\alpha_{g}(h^{(n_j)}_k)-h^{(n_j)}_k\| \leq 2/j
\]
for all $j\in\IN$. Let $e\in J_{\infty,\alpha}$ be the positive contraction defined by the sequence 
\[
e_k = \begin{cases} 0 &,\quad k<k_1 \\ h^{(n_j)}_k &,\quad k_j\leq k<k_{j+1}. \end{cases}
\]
It follows that
\[
\|[e,c_n]\| = \limsup_{k\to\infty}~ \|[e_k, c^{(n)}_k]\| = \inf_{j\in\IN}~\sup_{k\geq k_j}~ \|[h^{(n_j)}_k,c^{(n)}_k]\| \leq \inf_{j\in\IN}~ 2/j = 0
\]
and
\[
\|b_n-b_ne\| = \limsup_{k\to\infty}~\|b^{(n)}_k-b^{(n)}_ke_k\| = \inf_{j\in\IN}~\sup_{k\geq k_j}~\|b^{(n)}_k-b^{(n)}_kh^{(n_j)}_k\| =  0
\]
for all $n\in\IN$. Moreover, we have
\[
\|\alpha_{\infty,g}(e)-e\| = \limsup_{k\to\infty}~\|\alpha_g(e_k)-e_k\| = \inf_{j\in\IN}~\sup_{k\geq k_j}~\|\alpha_g(h^{(n_j)}_k)-h^{(n_j)}_k\| = 0
\]
for all $g\in G$.
By the choice of $b_n$ and $c_n$, we get that $e\in (J_{\infty,\alpha}\cap C')^{\alpha_\infty}$ with $eb=b$ for all $b\in J_{\infty,\alpha}\cap C$. This shows our claim.
\end{proof}

\begin{example} \label{ann-sigma}
Let $G$ be a second-countable, locally compact group. Let $A$ be a \cstar-algebra with an action $\alpha: G\curvearrowright A$. Let $C\subset A_{\infty,\alpha}$ be a separable, $\alpha_\infty$-invariant \cstar-subalgebra. Then $\ann(C,A_{\infty,\alpha})$ is a $G$-$\sigma$-ideal in $A_{\infty,\alpha}\cap C'$.
\end{example}
\begin{proof}
This is analogous to the proof of \ref{ideal-sequences}, so we omit the details.
\end{proof}

The following is the main technical reason for the usefulness of $G$-$\sigma$-ideals:

\begin{prop} \label{strong lift}
Let $G$ be a second-countable, locally compact group. Let $A$ be a \cstar-algebra with an action $\alpha: G\curvearrowright A$. Let $J\subset A$ be a $G$-$\sigma$-ideal. Let $B=A/J$ and $\beta: G\curvearrowright B$ the action induced by $\alpha$. Denote by $\pi: A\to B$ the (equivariant) quotient map. Then:
\begin{enumerate}[label=\textup{(\roman*)},leftmargin=*]
\item For every $\alpha$-invariant, separable \cstar-subalgebra $D\subset A$, the induced map $\pi: A\cap D'\to B\cap\pi(D)'$ is surjective and its kernel $J\cap D'$ is a $G$-$\sigma$-ideal in $A\cap D'$. \label{strong-lift:1}
\item For every $\beta$-invariant, separable \cstar-subalgebra $C\subset B$, there exists an equivariant c.p.c.\ order zero map $\psi: (C,\beta)\to (A,\alpha)$ such that $\pi\circ\psi=\id_C$. \label{strong-lift:2}
\end{enumerate}
\end{prop}
\begin{proof}
\ref{strong-lift:1}: The surjectivity follows already from the fact that $J$ is an ordinary $\sigma$-ideal and has nothing to do with equivariance, see \cite[1.6]{Kirchberg04}. The fact that $J\cap D'$ is a $G$-$\sigma$-ideal in $A\cap D'$ is trivial and follows directly from the definition of a $G$-$\sigma$-ideal and the separability of $D$.

\ref{strong-lift:2}: Let $D\subset A$ be a separable, $\alpha$-invariant \cstar-subalgebra with $\pi(D)=C$. Using the $G$-$\sigma$-ideal property, let $e\in (J\cap D')^\alpha$ be a positive contraction with $ed=d$ for all $d\in J\cap D$. Then we obtain an equivariant c.p.c.~order zero map $\kappa: (D,\alpha)\to (A,\alpha)$ via $\kappa(d)=(\eins-e)d$. By the choice of $e$, we have $\kappa(d)=0$ for all $d\in D\cap J$. Thus we get the desired equivariant c.p.c.\ order zero map by defining $\psi(c)=(\eins-e)d$, where $d\in D$ is any element with $\pi(d)=c$.
\end{proof}

\begin{rem}
Let $G$ be a locally compact group and let $\pi: (A,\alpha)\to (B,\beta)$ be an equivariant, surjective $*$-homomorphism between $G$-\cstar-dynamical systems. Following the example of Kirchberg from \cite[1.5]{Kirchberg04}, we will say that $\pi$ is strongly locally semi-split, if it satisfies condition \ref{strong lift}\ref{strong-lift:2}. In particular, the statement in \ref{strong lift} then tells us that an equivariant quotient map (and its restrictions to relative commutants of separable subsystems) is automatically strongly locally semi-split, if its kernel is a $G$-$\sigma$-ideal.
\end{rem}

In order to exploit this strong lifting result, we shall show some non-trivial facts about the continuous parts of (central) sequence algebras that makes the concept of $G$-$\sigma$-ideals applicable in these cases.

\begin{lemma} \label{F(A)-alternativ}
Let $G$ be a second-countable, locally compact group. Let $A$ be a \cstar-algebra with an action $\alpha: G\curvearrowright A$. Let $C$ be a separable, $\alpha_\infty$-invariant \cstar-subalgebra of $A_{\infty,\alpha}$. Then the canonical embedding
\[
A_{\infty,\alpha}\cap C'/\ann(C,A_{\infty,\alpha})\to F_{\alpha}(C,A_\infty)
\]
is an isomorphism. In other words, any $\tilde{\alpha}_\infty$-continuous element in $F(C,A_\infty)$ can be represented by an $\alpha_\infty$-continuous element in $A_\infty\cap C'$.
\end{lemma}
\begin{proof}
As $C$ is separable, we choose a dense sequence $c_n\in C$. Let each $c_n$ be represented by a bounded sequence $c^{(n)}_k\in A$. Applying \ref{fixed-approx-unit}, we may choose a countable approximate unit $h_n\in C$ with $\max_{g\in K} \|\alpha_{\infty,g}(h_{n})-h_{n}\|\to 0$ for every compact set $K\subset G$. Let each $h_n$ be represented by a bounded sequence $h^{(n)}_k\in A$.

Let $x\in F_\alpha(C,A_\infty)$ be an contraction and represent it by a bounded sequence $x_n\in A$. Let $K_0\subset G$ be a compact neighbourhood of the unit and consider a sequence $\eps_j\searrow 0$. Then $f: K_0\to [0,\infty),~ g\mapsto\|\tilde{\alpha}_g(x)-x\|$ is continuous. We have for all $n$ that $x\cdot h_{n} \in A_{\infty,\alpha}$ is well-defined, and in fact
\[
\|xh_{n}-\alpha_{\infty,g}(xh_n)\|\leq \|\tilde{\alpha}_{\infty,g}(x)-x\|+\|\alpha_{\infty,g}(h_n)-h_n\|
\] 
for all $g\in G$.
For every $j$, there is thus a large $n_j$ with
\[
\limsup_{k\to\infty}~\|x_kh^{(n)}_k-\alpha_g(x_kh^{(n)}_k)\| \leq f(g)+\eps_k
\]
for all $g\in K_0$ and $n\geq n_j$. Applying \ref{sequence-lift} and uniform continuity on compact sets, we see that therefore
\[
\limsup_{k\to\infty}~\max_{g\in K_0}~\|x_kh^{(n)}_k-\alpha_g(x_kh^{(n)}_k)\| \leq f(g)+\eps_j
\]
for all $n\geq n_j$. Since the $h_n\in C$ are an approximate unit, we also have
\[
h_nc_j-c_j = \limsup_{k\to\infty}~\|h^{(n)}_k c^{(j)}_k -c^{(j)}_k\| \stackrel{n\to\infty}{\longrightarrow} 0
\]
for all $j\in\IN$. So by making the $n_j$ bigger, if necessary, we may choose strictly increasing numbers $k_j$ so that
\[
\max_{i\leq j}~\|h^{(n_j)}_k c^{(j)}_k - c^{(j)}_k\|\leq\eps_j\quad\text{and}\quad \max_{g\in K_0}~\|x_k h^{(n_j)}_k-\alpha_g(x_k h^{(n_j)}_k)\| \leq f(g)+\eps_j
\]
for all $k\geq k_j$.
Considering the element $y\in A_\infty$ defined by the bounded sequence
\[
y_k = \begin{cases} 0 &,\quad k < k_1 \\
x_k h_k^{(n_j)} &,\quad k_j\leq k<k_{j+1} \end{cases}
\]
it follows that
\[
\|y-\alpha_{\infty,g}(y)\| \leq \limsup_{j\to\infty} \sup_{k\geq k_j} \|x_k-\alpha_g(x_kh_k^{(n_j)})\|\leq \limsup_{j\to\infty} f(g)+\eps_j = f(g)
\]
for all $g\in K_0$. Hence $y\in A_{\infty,\alpha}$. Moreover, we have
\[
\|[y,c_i]\| = \limsup_{j\to\infty}~ \sup_{k\geq k_j} \|[x_kh^{(n_j)}_k,c_i]\| \leq \limsup_{j\to\infty}~ \sup_{k\geq k_j} \|[h^{(n_j)}_k,c_i]\| = 0
\]
and
\[
\|xc_i-yc_i\| \leq \limsup_{j\to\infty} \sup_{k\geq k_j} \|x_k c_i-x_kh^{(n_j)}_kc_i\| \leq \limsup_{j\to\infty}~ \eps_j = 0
\]
for all $i\in\IN$. But since the $c_i\in C$ are dense, it follows that $y\in A_{\infty,\alpha}\cap A'$ and $x-y\in\ann(C,A_\infty)$. This finishes the proof.
\end{proof}

\begin{prop} \label{sequence-exact}
Let $G$ be a second-countable, locally compact group. Let
\[
\xymatrix{
0 \ar[r]& (J,\alpha) \ar[r] &  (A,\alpha) \ar[r]^\pi &  (B,\beta) \ar[r] & 0
}
\]
be a short exact sequence of $G$-\cstar-dynamical systems. If $\pi_\infty: A_\infty\to B_\infty$ denotes the induced map on the level of sequence algebras, then this induces an equivariant short-exact sequence
\[
\xymatrix{
0 \ar[r]& (J_{\infty,\alpha},\alpha_\infty) \ar[r] &  (A_{\infty,\alpha},\alpha_\infty) \ar[r]^{\pi_\infty} &  (B_{\infty,\beta},\beta_\infty) \ar[r] & 0.
}
\]
\end{prop}
\begin{proof}
Since $\pi_\infty$ is clearly equivariant, we have $\pi_\infty(A_{\infty,\alpha})\subset B_{\infty,\beta}$. The kernel of this restriction is $A_{\infty,\alpha}\cap\ker(\pi_\infty) = A_{\infty,\alpha}\cap J_{\infty}=J_{\infty,\alpha}$. So the only thing left to show is that $\pi_\infty(A_{\infty,\alpha})= B_{\infty,\beta}$.  Use \ref{fixed-approx-unit} and choose an approximate unit $(h_\lambda)_\lambda$ of $J$ that satisfies $\max_{g\in K}\|\alpha_g(h_\lambda)-h_\lambda\|\to 0$ for every compact set $K\subset G$.

 Let $x\in B_{\infty,\beta}$ be an element represented by a bounded sequence $x_n\in B$. By \ref{sequence-lift}, the function $f: G\to [0,\infty)$ given by $f(g)=\sup_{n\in\IN}\|\beta_g(x_n)-x_n\|$ is continuous. Choose a bounded sequence $y_n\in A$ with $\pi(y_n)=x_n$ for each $n$. Then we have
\[
\begin{array}{ccl}
 \|\beta_g(x_n)-x_n\| &=& \dst\lim_{\lambda\to\infty} \|(\eins-h_\lambda)(\alpha_g(y_n)-y_n)\| \\
 &=& \dst\lim_{\lambda\to\infty} \|\alpha_g\big((\eins-h_\lambda)y_n \big)-(\eins-h_\lambda)y_n\|
\end{array}
\]
for every $n\in\IN$, and this convergence is uniform on compact subsets of $G$. In particular, we may find a sequence of positive contractions $h_n\in J$ and a sequence $K_n\nearrow G$ such that
\[
\|\alpha_g((\eins-h_n)y_n)-(\eins-h_n)y_n\| \leq 1/n+f(g)
\]
for all $g\in K_n$. Considering the bounded sequence $z_n=(\eins-h_n)y_n\in A$, we then have $\pi(z_n)=x_n$ for all $n$. For the resulting element $z\in A_\infty$, we moreover have
\[
\|\alpha_g(z)-z\| = \limsup_{n\to\infty} \|\alpha_g(z_n)-z_n\| \leq f(g)
\]
for all $g\in G$. As $f$ was continuous, this shows $z\in A_{\infty,\alpha}$ and finishes the proof.
\end{proof}


\section{Extensions}

\noindent
In this section, we will show that given a semi-strongly self-absorbing $G$-action $\gamma$, the class of all separable, $\gamma$-absorbing $G$-\cstar-dynamical systems is closed under equivariant extensions, at least under a mild extra condition. In fact, this turns out to be true precisely when $\gamma$ is unitarily regular.

\begin{defi}
Let $D_1$ and $D_2$ be two unital \cstar-algebras. Define
\[
\CE(D_1,D_2)=\set{f\in\CC\big( [0,1], D_1\otimes_{\max} D_2 \big) \mid f(0)\in D_1\otimes\eins,\ f(1)\in\eins\otimes D_2 }
\]
and c.p.c.\ order zero maps $\mu_i: D_i\to\CE(D_1,D_2)$ for $i=1,2$ via
\[
\mu_i(d_i)(t) = \begin{cases} (1-t)(d_1\otimes\eins) &,\quad i=1 \\
t(\eins\otimes d_2) &,\quad i=2 \end{cases}
\]
for all $d_i\in D_i$ and $0\leq t\leq 1$.
\end{defi}

\begin{rem} \label{join}
Let $X$ and $Y$ be two compact Hausdorff spaces. If we apply the above construction to $D_1=\CC(X)$ and $D_2=\CC(Y)$, then there is a natural isomorphism
\[
\CE(\CC(X),\CC(Y)) \cong \CC(X\star Y),
\]
where $X\star Y$ is the join of $X$ and $Y$, i.e.\ the quotient space $[0,1]\times X\times Y/\sim$ by identifying $(0,x,y_1)\sim (0,x,y_2)$ and $(1,x_1,y)\sim (1,x_2,y)$. The c.p.c.\ order zero maps $\mu_1, \mu_2$ have no proper topological counterparts, however. For example, $\mu_1$ would need to correspond to the (not well-defined) map $X\star Y$ to $[0,1)\times X$ sending the equivalence class of a triple $(t,x,y)$ to $(t,x)$.
\end{rem}

The following fact follows from the characterization of c.p.c.\ order zero maps as $*$-homomorphisms from the cone \cite[4.1]{WinterZacharias09}, and applying \cite[5.2]{HirshbergWinterZacharias15} to the cones over $D_1$ and $D_2$. See also the proof of \cite[1.17]{Kirchberg04}, where this observation had previously appeared in a very similar context as this section.

\begin{prop}
Let $D_1$ and $D_2$ be two unital \cstar-algebras. Then the \cstar-algebra $\CE(D_1,D_2)$ and the maps $\mu_1, \mu_2$ satisfy the following universal property:

Whenever $B$ is a unital \cstar-algebra and $\eta_i: D_i\to B$ are c.p.c.\ order zero maps for $i=1,2$ with commuting ranges such that $\eta_1(\eins)+\eta_2(\eins)=\eins$, then there exists a unique unital $*$-homomorphism $\eta: \CE(D_1,D_2)\to B$ with $\eta_i = \eta\circ\mu_i$ for $i=1,2$.
\end{prop}

\begin{rem}
Let $G$ be a locally compact group. For two actions $\gamma_i: G\curvearrowright D_i$ for $i=1,2$, it follows from the above that we get a well-defined, point-norm continuous action $\CE(\gamma_1,\gamma_2): G\curvearrowright\CE(D_1,D_2)$ via $\CE(\gamma_1,\gamma_2)_g\circ\mu_i = \mu_i\circ\gamma_{i,g}$ for $i=1,2$ and $g\in G$. (In the continuous function picture, one simply applies the action $\gamma_1\otimes\gamma_2$ on each fibre.)
Then one immediately gets an equivariant analogue of the above universality result:
\end{rem}

\begin{prop} \label{G-universal}
Let $G$ be a locally compact group.
Let $D_1$ and $D_2$ be two unital \cstar-algebras, and let $\gamma_i: G\curvearrowright D_i$ for $i=1,2$ be two actions. Then the \cstar-dynamical system $\big( \CE(D_1,D_2), \CE(\gamma_1,\gamma_2) \big)$ and the equivariant maps $\mu_1, \mu_2$ satisfy the following universal property:

Whenever $B$ is a unital \cstar-algebra with an action $\beta: G\curvearrowright B$ and $\eta_i: (D_i,\gamma_i)\to (B,\beta)$ are equivariant c.p.c.\ order zero maps for $i=1,2$ with commuting ranges such that $\eta_1(\eins)+\eta_2(\eins)=\eins$, there exists a unique unital and equivariant $*$-homomorphism $\eta: \big( \CE(D_1,D_2), \CE(\gamma_1,\gamma_2) \big) \to (B,\beta)$ with $\eta_i = \eta\circ\mu_i$ for $i=1,2$.
\end{prop}

The following technical observation is most likely well-known among \cstar-algebraists. Since I could not find a reference for this specific statement, a proof is provided for the reader's convenience.

\begin{lemma} \label{approx units}
Let $A$ be a $\sigma$-unital \cstar-algebra and $J\subset A$ an ideal. Denote by $\pi: A\to A/J$ the quotient map. Let $h_n\in J$ be a countable approximate unit that is quasicentral relative to $A$, and let $a_n\in A$ be positive contractions such that the sequence $\pi(a_n)\in A/J$ is a countable approximate unit. Then there is a sequence $1\leq n_1 < n_2<\dots$ such that the sequence 
\[
h_{n_k}+(\eins-h_{n_k})^{1/2}a_k(\eins-h_{n_k})^{1/2}\in A
\] 
defines a countable approximate unit.
\end{lemma}
\begin{proof}
Pick a strictly positive contraction $b\in A$.
For every $n\in\IN$, we have
\[
\|(\eins-\pi(a_n))\pi(b)\| = \|\pi((\eins-a_n)b)\| = \lim_{k\to\infty} \|(\eins-h_k)(\eins-a_n)b\|
\]
Using that $h_n$ is quasicentral, we may thus choose the $n_k$ such that 
\[
\|[a_k,(\eins-h_{n_k})^{1/2}]\|\leq 1/k
\] 
and such that 
\[
\|(\eins-h_{n_k})(\eins-a_k)b\|\leq \|(\eins-\pi(a_k))\pi(b)\|+1/k.
\] 
It then follows that
\[
\begin{array}{lll}
\multicolumn{3}{l}{ \| \big( \eins-h_{n_k}-(\eins-h_{n_k})^{1/2}a_k(\eins-h_{n_k})^{1/2} \big) b\| } \\
\hspace{2mm} & \leq & 1/k + \| \big( \eins-h_{n_k}-(\eins-h_{n_k})a_k \big) b\| \\
&=& 1/k + \|(\eins-h_{n_k})(\eins-a_k)b\| \\
& \leq & 2/k + \|(\eins-\pi(a_k))\pi(b)\| \quad\stackrel{k\to\infty}{\longrightarrow} 0.
\end{array}
\]
\end{proof}

\begin{lemma}[\cf{Kirchberg04}{1.17}]
\label{ce-lemma}
Let $G$ be a second-countable, locally compact group. Let $D_1$ and $D_2$ be two separable, unital \cstar-algebras and $\gamma_i : G\curvearrowright D_i$ an action for $i=1,2$. Let
\[
\xymatrix{
0 \ar[r]& (J,\alpha) \ar[r] &  (A,\alpha) \ar[r]^\pi &  (B,\beta) \ar[r] & 0
}
\]
be a short exact sequence of separable $G$-\cstar-dynamical systems. Suppose that there exist unital and equivariant $*$-homomorphisms 
\[
\kappa: (D_1,\gamma_1)\to \big( F_{\infty,\beta}(B), \tilde{\beta}_\infty \big) \quad\text{and}\quad \kappa': (D_2,\gamma_2) \to \big( F_{\infty,\alpha}(J), \tilde{\alpha}_\infty \big).
\] 
Then there exists a unital and equivariant $*$-homomorphism 
\[
\Psi: \big( \CE(D_1,D_2), \CE(\gamma_1,\gamma_2) \big) \to \big( F_{\infty,\alpha}(A), \tilde{\alpha}_\infty \big)
\]
with $\tilde{\pi}_\infty\circ\Psi = \kappa\circ\ev_0$.
\end{lemma}
\begin{proof}
Let $\kappa: (D_1,\gamma_1) \to \big( F_{\infty,\beta}(B), \tilde{\beta}_\infty \big)$ be as above.
Combining \ref{F(A)-alternativ}, \ref{sequence-exact}, \ref{ideal-sequences} and \ref{strong lift}\ref{strong-lift:1}, we see that the two canonical maps
\[
A_{\infty,\alpha}\cap A'\to B_{\infty,\beta}\cap B' \quad\text{and}\quad B_{\infty,\beta}\cap B'\to F_{\infty,\beta}(B)
\]
are surjective. By \ref{ideal-sequences}, \ref{strong lift}\ref{strong-lift:1} and \ref{ann-sigma}, the kernels of both these maps are $G$-$\sigma$-ideals, respectively. So by \ref{strong lift}\ref{strong-lift:2}, they are strongly locally semi-split, and so is their composition. Let 
\[
\psi': (D_1,\gamma_1) \to (A_{\infty,\alpha}\cap A',\alpha_\infty)
\] 
be an equivariant c.p.c.\ order zero map lifting $\kappa$ in the sense that 
\[
\kappa(x)=\pi_\infty(\psi'(x))+\ann(B,B_\infty)
\] 
for all $x\in D_1$. Then the positive contraction $\pi_\infty(\psi'(\eins))\in \big( B_{\infty,\beta} \big)^{\beta_\infty}$ acts as a unit on $B$. 
By \ref{fixed-approx-unit}, we can choose a countable approximate unit $h_n\in J$ that is quasicentral relative to $A$ and satisfies $\max_{g\in K}~ \|\alpha_g(h_n)-h_n\|\to 0$ for all compact sets $K\subset G$. By passing to some suitable subsequence $h_{n_k}$, we can combine \ref{approx units} with a reindexation trick and obtain a positive contraction $h_0\in (J_{\infty,\alpha})^{\alpha_\infty}$ commuting with both $\psi'(D_1)$ and $A$, and such that 
\[
e_0=h_0+(\eins-h_0)\psi'(\eins)\in (A_{\infty,\alpha})^{\alpha_\infty}
\] 
acts as a unit on $A$. Consider the equivariant c.p.c.~order zero map 
\[
\psi=(\eins-h_0)\cdot \psi': (D_1,\gamma_1)\to (A_{\infty,\alpha}\cap A', \alpha_\infty),
\] 
which is still a lift for $\kappa$.

Let $C$ be the $\alpha_\infty$-invariant \cstar-subalgebra of $A_{\infty,\alpha}$ that is generated by $A\cup\psi(D_1)\cup\set{e_0}$. The ideal $J_{\infty,\alpha}\subset A_{\infty,\alpha}$ is a $G$-$\sigma$-ideal by \ref{ideal-sequences}. Find a positive contraction $h_1\in \big( J_{\infty,\alpha}\cap C' \big)^{\alpha_\infty}$ satisfying $h_1c=c$ for all $c\in J_{\infty,\alpha}\cap C$. Then $e_0-\psi(\eins)=h_0\in J_{\infty,\alpha}\cap C$, so $(\eins-h_1)(e_0-\psi(\eins))=0$. We obtain another equivariant c.p.c.~order zero map 
\[
\psi_1=(\eins-h_1)\cdot\psi: (D_1,\gamma_1)\to (A_{\infty,\alpha}\cap A',\alpha_\infty).
\] 
Note that we still have $\kappa(x)=\pi_\infty(\psi_1(x))+\ann(B,B_\infty)$ for all $x\in D_1$. Let $C_0\subset J_{\infty,\alpha}$ be the $\alpha_\infty$-invariant \cstar-algebra generated by $h_1C\cup\set{h_1}$.
 
By assumption, there exists a unital and equivariant $*$-homomorphism from $(D_2,\gamma_2)$ to $\big( F_{\infty,\alpha}(J), \tilde{\alpha}_\infty \big)$. By \ref{reindex1}, there is thus also a unital $*$-homomorphism from $(D_2,\gamma_2)$ to $\big( F_\alpha(C_0,J_{\infty}), \tilde{\alpha}_\infty \big)$. Combining \ref{ann-sigma}, \ref{F(A)-alternativ} and \ref{strong lift}\ref{strong-lift:2}, we see that the equivariant map
\[
J_{\infty,\alpha}\cap C'\to F_\alpha(C_0,J_\infty)
\]
is surjective and strongly locally semi-split.
So there is an equivariant c.p.c.~order zero map $\psi_2': (D_2,\gamma_2)\to \big( J_{\infty,\alpha}\cap C_0', \alpha_\infty \big)$ satisfying $\psi_2'(\eins)c=c$ for all $c\in C_0$. Consider the equivariant c.p.c.~order zero map $\psi_2: (D_2,\gamma_2)\to (J_{\infty,\alpha},\alpha_\infty)$ given by $\psi_2(x)=h_1\psi_2'(x)$ for all $x\in D_2$. By design, its image commutes with all of $C$, so both with $A$ and the image of $\psi_1$. Observe now that for every $a\in A$, we have
\[
\begin{array}{ccl}
\big( \psi_1(\eins)+\psi_2(\eins) \big)a &=& \big( (\eins-h_1)\psi(\eins)+h_1 \big) a \\
&=& \big( (\eins-h_1)e_0+h_1 \big)a ~=~ a.
\end{array}
\]
In particular, the two resulting equivariant c.p.c.~order zero maps
\[
\tilde{\psi}_i: (D_i, \gamma_i) \to \big( F_{\infty,\alpha}(A), \tilde{\alpha}_\infty \big),\quad \tilde{\psi}_i(x)=\psi_i(x)+\ann(A,A_\infty),\quad i=1,2
\]
have commuting ranges and satisfy $\tilde{\psi}_1(\eins)+\tilde{\psi}_2(\eins)=\eins$. By \ref{G-universal}, this induces a unital and equivariant $*$-homomorphism 
\[
\Psi: \big( \CE(D_1,D_2), \CE(\gamma_1,\gamma_2) \big) \to \big( F_{\infty,\alpha}(A), \tilde{\alpha}_\infty \big)
\]
with the desired property.
\end{proof}

\begin{nota}
For the rest of this section, let us abbreviate $D^{(2)}=\CE(D,D)$ and $\gamma^{(2)}=\CE(\gamma,\gamma)$ for a unital \cstar-dynamical system $(D,\gamma)$.
\end{nota}

Here comes the main result of this section:

\begin{theorem}[cf.~{\cite[4.5]{Kirchberg04} and \cite[4.3]{TomsWinter07}}] 
\label{ext}
Let $G$ be a second-countable, locally compact group. Let $\CD$ be a separable, unital \cstar-algebra and $\gamma: G\curvearrowright\CD$ a semi-strongly self-absorbing action. The following are equivalent:
\begin{enumerate}[label=\textup{(\roman*)},leftmargin=*]
\item The class of all separable, $\gamma$-absorbing $G$-\cstar-dynamical systems is closed under equivariant extensions. \label{ext:1}
\item The action $\gamma^{(2)}: G\curvearrowright\CD^{(2)}$ is $\gamma$-absorbing. \label{ext:2}
\item $\gamma$ is unitarily regular. (In fact, any $\delta<\eps/4$ can be inserted in \ref{unireg}.) \label{ext:3}
\item $\gamma$ has strongly asymptotically $G$-inner half-flip. \label{ext:4}
\item There exists a unital and equivariant $*$-homomorphism from $(\CD,\gamma)$ to $\big( \CD^{(2)}_{\infty,\gamma^{(2)}}, \gamma^{(2)}_\infty \big)$. \label{ext:5}
\end{enumerate}
\end{theorem}
\begin{proof}
\ref{ext:1}$\implies$\ref{ext:2}: Since we have a canonical equivariant short exact sequence of the form
\[
\xymatrix{
0 \ar[r] & \big( C_0(0,1)\otimes\CD, \id\otimes\gamma \big) \ar[r] & (\CD^{(2)},\gamma^{(2)}) \ar[r]^{\ev_0\oplus\ev_1} & (\CD\oplus\CD,\gamma\oplus\gamma) \ar[r] & 0, 
}
\]
this implication is trivial.

\ref{ext:2}$\implies$\ref{ext:3}: Let $\eps>0$ and $K\subset G$ be a compact set. Let $0<\delta_0<\eps/4$ be some number and $u,v\in\CU(\CD^\gamma_{\delta_0,K})$ two unitaries. Let $\eta>0$ be some number to be specified later. Apply \ref{approx-uniqueness} and choose a pair $(\delta, K_1)$ for $(A,\alpha)=(\CD,\gamma)$ and the triple $(\eta, K, F)$, where $F=\set{uvu^*v^*}$. Without loss of generality, we may assume $K\subset K_1$ and $\delta\leq\eta<1$. Since $\gamma^{(2)}$ is $\gamma$-absorbing, we can find\footnote{The existence follows, for example, from composing a $*$-homomorphism as in \ref{approx-inverse} with the second-factor embedding of $\CD$ into $\CD^{(2)}\otimes\CD$.} a $(\delta,K_1)$-approximately equivariant $*$-homomorphism $\psi: \CD\to\CD^{(2)}$. Since $\CD^{(2)}\subset\CC\big([0,1],\CD\otimes\CD\big)$ canonically, let us represent $\psi$ by a continuous family of $*$-homomorphisms $\psi_t: \CD\to\CD\otimes\CD$ for $0\leq t\leq 1$ with $\psi_0(\CD)\subset\CD\otimes\eins$ and $\psi_1(\CD)\subset\eins\otimes\CD$. Then $\psi_t$ is also $(\delta, K_1)$-approximately equivariant for every $0\leq t\leq 1$. By the choice of $(\delta, K_1)$, it follows that there exists $w\in\CU(\CD^\gamma_{\eta,K})$ with $\ad(w\otimes\eins)\circ\psi_0(uvu^*v^*)=_\eta uvu^*v^*\otimes\eins$. By replacing each map $\psi_t$ by $\ad(w\otimes\eins)\circ\psi_t$, we get that $\psi$ is $(3\eta, K)$-approximately equivariant and satisfies $\psi_0(uvu^*v^*)=_\eta uvu^*v^*\otimes\eins$. Note that each $*$-homomorphism $\psi_t: \CD\to\CD\otimes\CD$ must then also be $(3\eta, K)$-approximately equivariant and we still have $\psi_0(\CD)\subset\CD\otimes\eins$ and $\psi_1(\CD)\subset\eins\otimes\CD$. 

Consider the unitary path $z_t: [0,1]\to\CU(\CD\otimes\CD)$ given by $z(t)=\psi_0(u)\psi_{1-t}(v)\psi_0(u^*)\psi_{1-t}(v^*)$. Then $z(0)=\eins$ and $z(1)=\psi_0(uvu^*v^*)$. Moreover, since each $\psi_t$ is $(3\eta, K)$-approximately equivariant and $u,v\in\CU(\CD^\gamma_{\delta_0,K})$, it follows that $\psi_t(u), \psi_t(v)\in \CU\big( (\CD\otimes\CD)^{(\gamma\otimes\gamma)}_{\delta_0+3\eta,K} \big)$ for every $0\leq t\leq 1$. So we have $z(t)\in\CU\big( (\CD\otimes\CD)^{(\gamma\otimes\gamma)}_{4\delta_0+12\eta,K} \big)$ for all $0\leq t\leq 1$, since it is a product of four such elements. Now apply \ref{approx-inverse} and find a unital and $(\eta, K)$-approximately equivariant $*$-homomorphism $\kappa: \CD\otimes\CD\to \CD$ satisfying $\kappa(uvu^*v^*\otimes\eins)=_\eta uvu^*v^*$. Then the unitary path $y=\kappa\circ z: [0,1]\to\CU(\CD)$ satisfies $y(t)\in\CU(\CD^\gamma_{4\eps+13\eta,K})$ for all $0\leq t\leq 1$, and moreover
\[
y(0)=\eins,\quad y(1)=\kappa(\psi_0(uvu^*v^*)) =_\eta \kappa(uvu^*v^*) =_\eta uvu^*v^*.
\]
It follows that $uvu^*v^*$ is $2\eta$-close to the unitary $y(1)\in\CU_0(\CD^\gamma_{4\delta_0+13\eta,K})$. Since $\eta<1$, it follows by \ref{rem:1-unitaries}\ref{1-unitaries:3} that $uvu^*v^*\in\CU_0(\CD^\gamma_{4\delta_0+17\eta, K})$. Since $\eta$ may be chosen independently of $\delta_0$, we may choose $\eta\leq\frac{\eps-4\delta_0}{17}$, and then it follows that $4\delta_0+17\eta\leq\eps$. This shows our claim.

\ref{ext:3}$\implies$\ref{ext:4}: This follows directly from \ref{asue}.

\ref{ext:4}$\implies$\ref{ext:5}: Pick sequences $\eps_n\searrow 0$ and $K_n\nearrow G$. As $(\CD,\gamma)$ has strongly asymptotically $G$-inner half-flip, there exist in particular unitary paths $w_n: [0,1)\to\CU\big( (\CD\otimes\CD)^{(\gamma\otimes\gamma)}_{\eps_n/2, K_n} \big)$ with $w_n(0)=\eins$ and such that
\[
\lim_{t\to 1} w_n(t)(x\otimes\eins)w_n(t)^* = \eins\otimes x\quad\text{for all}~x\in\CD~\text{and}~n\in\IN.
\]
Consider the sequence of $*$-homomorphisms
\[
\psi_n: \CD\to\CC\big([0,1], \CD\otimes\CD \big),\quad \psi_n(x)(t) = \begin{cases} w_n(t)(x\otimes\eins)w_n(t)^* &,\quad 0\leq t<1 \\ \eins\otimes x &,\quad t=1. \end{cases}
\] 
Then by choice of the paths $w_n$, each $*$-homomorphism $\psi_n$ has image in $\CD^{(2)}$ and is $(\eps_n, K_n)$-approximately equivariant with respect to $\gamma$ and $\gamma^{(2)}$. By the choice of $\eps_n$ and $K_n$, this implies the existence of a unital and equivariant $*$-homomorphism from $(\CD,\gamma)$ to $\big( \CD^{(2)}_{\infty,\gamma^{(2)}}, \gamma^{(2)}_\infty \big)$.

\ref{ext:5}$\implies$\ref{ext:1}: Let
\[
\xymatrix{
0 \ar[r]& (J,\alpha) \ar[r] &  (A,\alpha) \ar[r]^\pi &  (B,\beta) \ar[r] & 0
}
\]
be a short exact sequence of separable $G$-\cstar-dynamical systems with $\alpha|_J\cc \alpha|_J\otimes\gamma$ and $\beta\cc\beta\otimes\gamma$. Then there exist unital and equivariant $*$-homomorphisms
\[
(\CD,\gamma)\to \big( F_{\infty,\beta}(B), \tilde{\beta}_\infty \big) \quad\text{and}\quad (\CD,\gamma)\to \big( F_{\infty,\alpha}(J), \tilde{\alpha}_\infty \big).
\] 
By \ref{ce-lemma}, it follows that there exists a unital and equivariant $*$-homomorphism from $(\CD^{(2)},\gamma^{(2)})$ to $\big( F_{\infty,\alpha}(A), \tilde{\alpha}_\infty \big)$. By condition \ref{ext:5}, we can form sequence algebras on both sides and obtain a unital and equivariant $*$-homomorphism $\psi: (\CD,\gamma)\to\Big( \big( F_{\infty,\alpha}(A)\big)_{\infty,\tilde{\alpha}_\infty}, (\tilde{\alpha}_\infty)_\infty \Big)$. 

Note that we have a map $A \otimes_{\max} F_{\infty,\alpha}(A)_{\infty,\tilde{\alpha}_\infty} \to \big( A \otimes_{\max} F_{\infty,\alpha}(A) \big)_{\infty, (\alpha\otimes\tilde{\alpha}_\infty)}$ in a natural way. Denote by $\theta: A \otimes_{\max} F_{\infty,\alpha}(A) \to A_{\infty,\alpha}$ the (equivariant) $*$-homomorphism induced by multiplication. We can use this to obtain a commutative diagram of the form
\[
\xymatrix{
A \ar[rrrr] \ar[rd]_{\id_A\otimes \eins} & & & & (A_{\infty,\alpha})_{\infty,\alpha_\infty} \\
	 & A\otimes \CD  \ar[rd]_{\id_A\otimes\psi} \ar@{-->}[rrru]& & & \\
	 & & \big( A \otimes_{\max} F_{\infty,\alpha}(A) \big)_{\infty, (\alpha\otimes\tilde{\alpha}_\infty)} \ar[rruu]_{\theta_\infty} & &
}
\]
Note that all the $*$-homomorphisms in this diagram are equivariant with respect to the obvious actions. Applying \cite[3.6, 4.2]{BarlakSzabo15}, it follows that there is also a unital and equivariant $*$-homomorphism from $(\CD,\gamma)$ to $\big( F_{\infty,\alpha}(A), \tilde{\alpha}_\infty \big)$. Thus $\alpha\cc\alpha\otimes\gamma$ and this finishes the proof.
\end{proof}


\section{Concluding remarks}

\noindent
In view of \ref{ext}, let us ask:

\begin{question} \label{Q-K1}
Is every semi-strongly self-absorbing action automatically unitarily regular? Is already one of conditions \ref{jiang-su-k1}\ref{jiang-su-k1:1} or \ref{jiang-su-k1}\ref{jiang-su-k1:2} automatic for every semi-strongly self-absorbing action? 
\end{question}

As every equivariantly $\CZ$-stable action satisfies all of these properties, we shall treat some (non-)examples of equivariantly $\CZ$-stable, semi-strongly self-absorbing actions and shed light on the difficulty of this question. First, we shall discuss a case where equivariant $\CZ$-stability is indeed automatic:

\begin{rem} 
Let $G$ be a countable, amenable group. Let $\CD$ be a separable, unital \cstar-algebra and $\gamma: G\curvearrowright\CD$ a semi-strongly self-absorbing action. If $\CD$ is infinite, then $\CD$ is a Kirchberg algebra. It is shown in \cite{Szabo16_K} that all $G$-actions on Kirchberg algebras are equivariantly $\CO_\infty$-absorbing, so this is in particular the case for $\gamma$. 

If $\CD$ is finite, then it has unique trace. As $\gamma$ is semi-strongly self-absorbing, we have that for all $g\in G$, the automorphism $\gamma_g\in\Aut(\CD)$ is either trivial or strongly outer. If $N=\ker(\gamma)$ is the kernel of $\gamma$, then we get a well-defined action $\bar{\gamma}: G/N\curvearrowright\CD,~ \bar{\gamma}_{gN}=\gamma_g$ induced by $\gamma$, which is pointwise strongly outer. If $G$ has property (Q) in Matui-Sato's sense \cite[2.4]{MatuiSato14}, then so does $G/N$, and thus $\bar{\gamma}\cc\bar{\gamma}\otimes\id_\CZ$ by \cite[4.10]{MatuiSato14}. But then it follows trivially that $\gamma$ was equivariantly $\CZ$-stable to begin with.

Now every Weiss-tileable group has property (Q) by \cite[5.10]{Wang14} and there is no example of a countable, amenable group that is known not to be Weiss-tileable. In particular, we do not know of any example of a semi-strongly self-absorbing action of a countable, amenable group that is not equivariantly $\CZ$-stable. As we will see further below, the situation is rather different for non-discrete acting groups, where such examples do exist.
\end{rem}

The following elementary construction yields the existence of faithful and strongly self-absorbing actions of large classes of locally compact groups on strongly self-absorbing \cstar-algebras. Although this was hinted at in \cite[Section 5]{Szabo15ssa}, it was not made explicit.

\begin{prop}
Let $D$ be a separable, unital \cstar-algebra with approximately inner flip. Let $G$ be a second-countable, locally compact group and $u: G\to\CU(\CD)$ a continuous unitary representation. Then the action
\[
\bigotimes_\IN\ad(u): G\curvearrowright\bigotimes_\IN D
\]
is strongly self-absorbing.
\end{prop}
\begin{proof}
Let $v_n\in D\otimes D$ be a sequence of unitaries approximately implementing the flip automorphism. Given a compact set $K\subset G$, the subset $\set{u_g\otimes u_g \mid g\in K}\subset\CD\otimes\CD$ is also compact, and thus
\[
\max_{g\in K}~ \|\ad(u_g\otimes u_g)(v_n)-v_n\| = \max_{g\in K}~\|v_n(u_g\otimes u_g)v_n^*-u_g\otimes u_g\| \stackrel{n\to\infty}{\longrightarrow} 0.
\]
This shows that the system $(D,\ad(u))$ has approximately $G$-inner flip. The infinite tensor product of this system is then strongly self-absorbing by \cite[3.3]{Szabo15ssa}.
\end{proof}

\begin{example}
The circle action
\[
\gamma:=\bigotimes_\IN\ad\matrix{1 & 0 \\ 0 & z}: \IT\curvearrowright \bigotimes_\IN M_2 = M_{2^\infty}
\]
is strongly self-absorbing by the above, but not equivariantly $\CZ$-stable. If it were, then the fixed-point algebra $(M_{2^\infty})^\gamma$ would need to be $\CZ$-stable, as it embeds into the crossed product as a corner. However, this fixed-point algebra is known as the GICAR algebra (see \cite[Section 5]{Bratteli72}) and is well-known to have characters. On the other hand, the GICAR algebra is AF, and so all unitaries in the fixed point algebra $\big( (M_{2^\infty})_{\infty,\gamma}\big)^{\gamma_\infty} = \big( (M_{2^\infty})^\gamma \big)_\infty$ are homotopic to $\eins$. This fixed-point algebra is thus $K_1$-injective. In particular, $\gamma$ is unitarily regular by \ref{jiang-su-k1}, but not equivariantly $\CZ$-stable. 
\end{example}

\begin{rem}
Considering the above example, it stands to reason that Question \ref{Q-K1} might be very difficult in general. In applications, it is usually unclear how to verify $K_1$-injectivity for a \cstar-algebra (or its sequence algebra) without appealing to $\CZ$-stability or pure infiniteness. Already in the classical theory, nobody has (so far) provided a direct proof of the fact that strongly self-absorbing \cstar-algebras are $K_1$-injective. The only known proof makes use of Winter's result \cite{Winter11} that they are in fact $\CZ$-stable. The same is true for the fact that in every strongly self-absorbing \cstar-algebra $\CD$, the commutator subgroup of $\CU(\CD)$ is in $\CU_0(\CD)$. Given the above example, trying to appeal automatically to equivariant $\CZ$-stability is no suitable approach to solve Question \ref{Q-K1} for (semi-)strongly self-absorbing actions in general.
\end{rem}

\begin{rem}
Let us observe that \ref{ce-lemma} can be used to easily obtain extension results about certain invariants of group actions, such as Rokhlin dimension. So let $G$ be a second-countable, locally compact group and $\alpha: G\curvearrowright A$ an action on a separable \cstar-algebra. Let $J\subset A$ be an $\alpha$-invariant ideal and $B=A/J$ the quotient, with induced action $\beta$. Let $\gamma: G\curvearrowright D$ be an action on a separable, unital \cstar-algebra.

Suppose that there exist equivariant c.p.c.\ order zero maps
\[
\phi_0,\dots,\phi_{d_1}: (D,\gamma) \to \big( F_{\infty,\alpha}(J), \tilde{\alpha}_\infty \big)
\]
and
\[
\psi_0,\dots,\psi_{d_2}: (D,\gamma) \to \big( F_{\infty,\beta}(B), \tilde{\beta}_\infty \big)
\]
satisfying
\[
\eins=\phi_0(\eins)+\dots+\phi_{d_1}(\eins)\quad\text{and}\quad\eins=\psi_0(\eins)+\dots+\psi_{d_1}(\eins).
\]
Applying \ref{ce-lemma} and \ref{G-universal} to 
\[
D_1=\mathrm{C}^*(\phi_i(D) \mid 0\leq i\leq d_1) \quad\text{and}\quad D_2=\mathrm{C}^*(\psi_i(D) \mid 0\leq i\leq d_2)
\] 
with the obvious actions, it follows that there exist equivariant c.p.c.\ order zero maps
\[
\kappa_0,\dots,\kappa_{1+d_1+d_2}: (D,\gamma)\to \big( F_{\infty,\alpha}(A), \tilde{\alpha}_\infty \big)
\]
with
\[
\eins=\kappa_0(\eins)+\dots+\kappa_{1+d_1+d_2}(\eins).
\]
Moreover, if the collections $\set{\phi_i}_i$ and $\set{\psi_i}_i$ have pairwise commuting ranges, then we may require that the maps $\set{\kappa_i}_i$ have pairwise commuting ranges as well.
Let us list some special cases and implications of this observation:
\begin{enumerate}[label={(\roman*)},leftmargin=*]
\item Let $H\subset G$ be a closed and cocompact subgroup. Consider $D=\CC(G/H)$ equipped with the canonical $G$-shift. Then we obtain
\[
\dimrok(\alpha, H) \leq \dimrok(\alpha|_J, H)+\dimrok(\beta,H)+1
\] 
and
\[
\dimrokc(\alpha, H) \leq \dimrokc(\alpha|_J, H)+\dimrokc(\beta,H)+1,
\]
where $\dimrok(\alpha, H)$ denotes the Rokhlin dimension of $\alpha$ relative to $H$, as defined in \cite[Section 5]{HirshbergSzaboWinterWu16}. The expression $\dimrokc(\alpha, H)$ denotes the analogous notion where one requires commuting towers.
\item If $G=\IR$, then this implies
\[
\dimrok(\alpha) \leq \dimrok(\alpha|_J)+\dimrok(\beta)+1
\] 
and
\[
\dimrokc(\alpha) \leq \dimrokc(\alpha|_J)+\dimrokc(\beta)+1,
\] 
where these values denote the Rokhlin dimension of $\alpha$ as a flow, with or without commuting towers, in the sense of \cite[Section 2+5]{HirshbergSzaboWinterWu16}.
\item If $G$ is discrete and residually finite, then we get
\[
\dimrok(\alpha) \leq \dimrok(\alpha|_J)+\dimrok(\beta)+1
\] 
and
\[
\dimrokc(\alpha) \leq \dimrokc(\alpha|_J)+\dimrokc(\beta)+1,
\] 
where these values denote the Rokhlin dimension of $\alpha$, with or without commuting towers, in the sense of \cite[Section 4+9]{SzaboWuZacharias15}. This recovers and generalizes an extension result for Rokhlin dimension of finite group and integer actions \cite[2.10, 3.5]{HirshbergPhillips15} due to Hirshberg-Phillips.
\item Assume that $G$ is compact. Let $\theta: G\curvearrowright X$ be a continuous action on a metrizable, compact space. Suppose that both $\alpha|_J$ and $\beta$ have the $(X,\theta)$-Rokhlin property in the sense of \cite[1.6]{HirshbergPhillips15}. (Note that although this definition required unitality of the underlying \cstar-algebra, we can extend it by using the corrected central sequence algebra, if the underlying \cstar-algebra is non-unital.) Then it follows from \ref{join} that $\alpha$ has the $(X\star X,\theta\star\theta)$-Rokhlin property, where $X\star X$ is the join of $X$ with itself and $\theta\star\theta$ is the diagonal action induced by $\theta$. If one extends the definition of the $(X,\theta)$-Rokhlin property further to include locally compact group actions (the original definition \cite[1.6]{HirshbergPhillips15} of Hirshberg-Phillips assumes compactness of $G$), then the same statement remains true in that context.
\end{enumerate}
\end{rem}


\bibliographystyle{gabor}
\bibliography{master}

\begin{thebibliography}{10}
\providecommand{\url}[1]{\texttt{#1}}
\providecommand{\urlprefix}{URL }

\bibitem{BarlakSzabo15}
S.~Barlak, G.~Szab{\'o}: Sequentially split $*$-homomorphisms between
  \cstar-algebras  (2015).
\newblock \urlprefix\url{http://arxiv.org/abs/1510.04555v2}.

\bibitem{Blackadar15}
B.~Blackadar: The homotopy lifting theorem for semiprojective \cstar-algebras.
\newblock Math. Scand., to appear  (2015).
\newblock \urlprefix\url{http://arxiv.org/abs/1207.1909v4}.

\bibitem{Bratteli72}
O.~Bratteli: Inductive limits of finite dimensional \cstar-algebras.
\newblock Trans. Amer. Math. Soc. 171 (1972), pp. 195--234.

\bibitem{DadarlatWinter09}
M.~Dadarlat, W.~Winter: The {$KK$}-theory of strongly self-absorbing
  \cstar-algebras.
\newblock Math. Scand. 104 (2009), no.~1, pp. 95--107.

\bibitem{GongLinNiu15}
G.~Gong, H.~Lin, Z.~Niu: {C}lassification of simple amenable {$\CZ$}-stable
  \cstar-algebras  (2015).
\newblock \urlprefix\url{http://arxiv.org/abs/1501.00135}.

\bibitem{HirshbergPhillips15}
I.~Hirshberg, N.~C. Phillips: Rokhlin dimension: obstructions and permanence
  properties.
\newblock Doc. Math. 20 (2015), pp. 199--236.

\bibitem{HirshbergSzaboWinterWu16}
I.~Hirshberg, G.~Szab{\'o}, W.~Winter, J.~Wu: Rokhlin dimension for flows
  (2016).
\newblock In preparation.

\bibitem{HirshbergWinterZacharias15}
I.~Hirshberg, W.~Winter, J.~Zacharias: Rokhlin dimension and \cstar-dynamics.
\newblock Comm. Math. Phys. 335 (2015), pp. 637--670.

\bibitem{Izumi10}
M.~Izumi: Group {A}ctions on {O}perator {A}lgebras.
\newblock Proc. Intern. Congr. Math.  (2010), pp. 1528--1548.

\bibitem{jiang}
X.~Jiang: Nonstable {$K$}-theory for {$\CZ$}-stable \cstar-algebras  (1997).
\newblock \urlprefix\url{http://arxiv.org/abs/math/9707228v1}.

\bibitem{Kasparov88}
G.~G. Kasparov: Equivariant {$KK$}-theory and the {N}ovikov conjecture.
\newblock Invent. Math. 91 (1988), pp. 147--201.

\bibitem{KirchbergC}
E.~Kirchberg: The {C}lassification of {P}urely {I}nfinite \cstar-{A}lgebras
  {U}sing {K}asparov's {T}heory (2003).
\newblock Preprint.

\bibitem{Kirchberg04}
E.~Kirchberg: Central sequences in \cstar-algebras and strongly purely infinite
  algebras.
\newblock Operator Algebras: The Abel Symposium 1 (2004), pp. 175--231.

\bibitem{KirchbergPhillips00}
E.~Kirchberg, N.~C. Phillips: Embedding of exact \cstar-algebras and continuous
  fields in the {C}untz algebra {$\CO_2$}.
\newblock J. reine angew. Math. 525 (2000), pp. 17--53.

\bibitem{KirchbergRordam12}
E.~Kirchberg, M.~R{\o}rdam: Central sequence \cstar-algebras and tensorial
  absorption of the {J}iang-{S}u algeba.
\newblock J. reine angew. Math. 695 (2014), pp. 175--214.

\bibitem{Liao15}
H.~C. Liao: A {R}okhlin type theorem for simple \cstar-algebras of finite
  nuclear dimension.
\newblock J. Funct. Anal. 270 (2016), no.~10, pp. 3675--3708.

\bibitem{MatuiSato14}
H.~Matui, Y.~Sato: {$\CZ$}-stability of crossed products by strongly outer
  actions {II}.
\newblock Amer. J. Math. 136 (2014), pp. 1441--1497.

\bibitem{PackerRaeburn89}
J.~A. Packer, I.~Raeburn: Twisted crossed products of \cstar-algebras.
\newblock Math. Proc. Cambridge Philos. Soc. 106 (1989), no.~2, pp. 293--311.

\bibitem{Phillips00}
N.~C. Phillips: A classification theorem for nuclear purely infinite simple
  \cstar-algebras.
\newblock Doc. Math. 5 (2000), pp. 49--114.

\bibitem{Szabo15ssa}
G.~Szab\'{o}: Strongly self-absorbing {$\mathrm{C}^*$}-dynamical systems.
\newblock Trans. Amer. Math. Soc., to appear  (2015).
\newblock \urlprefix\url{http://arxiv.org/abs/1509.08380v1}.

\bibitem{Szabo16_K}
G.~Szab{\'o}: Equivariant tensorial absorption for actions of discrete amenable
  groups on {K}irchberg algebras (2016).
\newblock In preparation.

\bibitem{SzaboWuZacharias15}
G.~Szab{\'o}, J.~Wu, J.~Zacharias: Rokhlin dimension for actions of residually
  finite groups (2015).
\newblock \urlprefix\url{http://arxiv.org/abs/1408.6096v3}.

\bibitem{TomsWinter07}
A.~S. Toms, W.~Winter: Strongly self-absorbing \cstar-algebras.
\newblock Trans. Amer. Math. Soc. 359 (2007), no.~8, pp. 3999--4029.

\bibitem{Wang14}
Q.~Wang: Tracial {R}okhlin property for actions of amenable groups.
\newblock Rocky Mountain J. Math., to appear  (2016).
\newblock \urlprefix\url{http://arxiv.org/abs/1410.8170v1}.

\bibitem{Winter11}
W.~Winter: Strongly self-absorbing \cstar-algebras are {$\CZ$}-stable.
\newblock J. Noncomm. Geom. 5 (2011), no.~2, pp. 253--264.

\bibitem{Winter14Lin}
W.~Winter: Localizing the {E}lliott conjecture at strongly self-absorbing
  \cstar-algebras, with an appendix by {H.} {L}in.
\newblock J. reine angew. Math. 692 (2014), pp. 193--231.

\bibitem{WinterZacharias09}
W.~Winter, J.~Zacharias: Completely positive maps of order zero.
\newblock M{\"u}nster J. Math. 2 (2009), pp. 311--324.

\end{thebibliography}

\end{document}